\documentclass[12pt,a4paper,twoside]{article}
\usepackage[colorlinks=true,linkcolor=blue,citecolor=blue]{hyperref}
\usepackage{amsmath,mathrsfs,multicol,amsfonts,mathtools,amsthm, bm,fancyhdr,lastpage,xcolor,titlesec,cleveref}
\usepackage[english]{babel}
\usepackage{nicematrix}
\usepackage{amssymb}
\definecolor{myeditcolor}{named}{blue}
\usepackage[bottom,multiple]{footmisc}
\usepackage[top=2.5cm,bottom=2.5cm,left=2.5cm,right=2.5cm]{geometry}
\titleformat{\section}[block]{\bfseries\large}{\thesection. }{2pt}{}
\theoremstyle{definition}
\newtheorem{theorem}{Theorem}[section]
\newtheorem{definition}[theorem]{Definition}
\newtheorem{proposition}{Proposition}[section]
\newtheorem{corollary}[theorem]{Corollary}
\newtheorem{lemma}[theorem]{Lemma}

\newtheorem{remark}[theorem]{Remark}
\newcommand{\edit}[1]{\textcolor{myeditcolor}{#1}}

\begin{document}
\thispagestyle{empty}
\begin{center}
\noindent\textbf{{\Large Isometry groups of simply connected nonunimodular Lie groups of dimension four}}
\end{center}
\begin{center} \noindent Youssef Ayad\footnote{corresponding author: youssef.ayad@edu.umi.ac.ma}\\
	\small{\textit{Faculty of sciences, Moulay Ismail University of Meknes}}\\
	\small{\textit{B.P. 11201, Zitoune, Meknes}, Morocco}
\end{center}
\begin{center}
\textbf{Abstract}
\end{center}
\begin{center}
\hspace*{3mm}    For each left-invariant Riemannian metric on simply connected nonunimodular Lie groups of dimension four, we determine the full group of isometries.
\end{center}
\begin{center}
\end{center}
\begin{center}
\end{center}
\section{Introduction}
\hspace*{3mm}     This paper is devoted to solving the following problem: For each left-invariant Riemannian metric on a simply connected, nonunimodular Lie group of dimension four, determine the full group of isometries. The latter depends on the left-invariant Riemannian metrics.\\
Let $G$ be a connected and simply connected Lie group with Lie algebra $\mathfrak{g}$. An isometry of $G$ is a diffeomorphism $\theta$ of $G$ such that the pullback of the metric $g$ by $\theta$ is equal to $g$, i.e. $g(u, v) = g(\theta_{\ast}^{-1}u, \theta_{\ast}^{-1}v) \,\, \forall u, v\in \mathfrak{g}$, where $\theta_{\ast}$ denotes the differential of $\theta$ at the identity element $e$ of $G$ (this definition is used because to describe the group of isometries, it suffices to examine the isotropy subgroup at the identity element, as will be explained below). Isometries are crucial in both mathematics and physics because they preserve essential concepts, including geodesics, the Levi-Civita connection, Ricci curvature, and others. This preservation is a fundamental feature in both disciplines.\\
The set of all isometries of $G$ with respect to the left-invariant Riemannian metric $g$, denoted $\operatorname{Isom}(G, g)$, forms a Lie group under the compact-open topology and acts transitively on $G$ \cite{myers1939group}. The isotropy subgroup at the identity element $e$ of $G$ is denoted by $\operatorname{Isom}(G, g)_{e}$, and it consists of the isometries of $G$ that fix the identity element $e$. Thus, we have the following decomposition: $\operatorname{Isom}(G, g) = L(G) \cdot \operatorname{Isom}(G, g)_{e} \cong G \cdot \operatorname{Isom}(G, g)_{e}$, where $L(G)$ is the subgroup of $\operatorname{Isom}(G, g)$ consisting of left translations on $G$, which is identified with $G$. In fact, if $\theta$ is an element of $\operatorname{Isom}(G, g)$ such that $\theta(e) = p \in G$, then $\theta$ decomposes as: $\theta = L_{p} \circ (L_{p^{-1}} \circ \theta) \quad \text{where}\,\, L_{p^{-1}} \circ \theta \in \operatorname{Isom}(G, g)_{e}$.\\
The product $L(G) \cdot \operatorname{Isom}(G, g)_{e}$ is not, in general, a semidirect product; it depends on whether $L(G)$ is normal in $\operatorname{Isom}(G, g)$. In \cite{wolf1962locally,wilson1982isometry}, it was proved that if $G$ is nilpotent, then $\operatorname{Isom}(G, g) \cong G \rtimes \operatorname{Aut}(G)_{g}$ where $\operatorname{Aut}(G)_{g} = \left\lbrace A \in \operatorname{Aut}(G) / A^tgA = g\right\rbrace$ is the group of isometric automorphisms of $G$ ($A^t$ is the matrix transpose of $A$). The isometry groups of left-invariant metrics on $4$-dimensional nilpotent Lie groups are studied in \cite{ayad2024isometry,vsukilovic2017isometry}.\\
\hspace*{3mm}  The study of the isometry group of left-invariant Riemannian metrics on Lie groups has been a key topic in differential geometry and Lie group theory. This problem was first systematically addressed by Takahashi and Ochiai in their influential paper \cite{ochiai1976group}. In their work, they focus on the case where the Lie group $G$ is simple. They showed that in this case, any isomtery in the connected component of the identity $\operatorname{Isom}(G, g)_0$ decomposes into a left translation and a right translation. That is $\operatorname{Isom}(G, g)_0 \subset L(G)R(G)$.\\
Later, Shin in \cite{shin1997isometry} focused on three-dimensional unimodular Lie groups, contributing to the understanding of the isometry groups in this dimension. In his work, he specifically examined the behavior of isometries and their structure on three-dimensional unimodular Lie groups, providing valuable insights into the problem in low dimensions.\\
The research continued with significant contributions from Ha and Lee \cite{ha2012isometry}, who studied the isometry groups of left-invariant metrics on a broader class of Lie groups. In particular, they made important remarks on three-dimensional Lie groups and completely solved the problem in this case, offering a thorough and insightful characterization of isometry groups for three-dimensional unimodular Lie groups.\\
In contrast, Reggiani and Cosgaya \cite{ana2022isometry} turned their attention to nonunimodular three-dimensional Lie groups. They used the fact that the differential of an isometry fixing the identity element $e$ of $G$ preserves the eigenspaces of the Ricci operator of $(G,g)$. They provided a complete description of the isometry group, which plays an essential role in determining the index of symmetry of the Riemannian Lie group.\\
Additionally, \cite{ayad2024isometry,aitbenhaddou2024isometry} addressed the problem of determining the isometry group of left-invariant Riemannian metrics on four-dimensional unimodular Lie groups.\\
\hspace*{3mm}   The problem of describing the isometry groups of left-invariant Riemannian metrics on Lie groups is closely related to the classification of these metrics up to automorphism. This is because when two metrics are equivalent under an automorphism, their isometry groups are conjugate. The classification of left-invariant Riemannian metrics on Lie groups has been addressed by Ha and Lee \cite{ha2009left} for dimension three, by Van Thuong \cite{van2017metrics} for dimension four in the unimodular case, and by \cite{aitbenhaddou2025classification} for dimension four in the nonunimodular case.\\
\hspace*{3mm} In this paper, we will use the classification of left-invariant Riemannian metrics provided in \cite{aitbenhaddou2025classification} and give their associated isometry groups. This work provides a complete description of the isometry groups for nonunimodular four-dimensional Lie groups, which is essential for understanding the underlying symmetries of these Lie groups. By correlating the classification of metrics with the structure of their isometry groups, we offer valuable insights into the geometry and symmetry properties of these groups, contributing to the broader understanding of Riemannian geometry and its applications in mathematical physics.
\section{Preliminaries}
Since our Lie group $G$ is simply connected, then $\operatorname{Aut}(G) \cong \operatorname{Aut}(\mathfrak{g})$ \cite{warner} and we have an action of this group on the set $\mathscr{L}$ of all left invariant Riemannian metrics on $G$ given by
$\operatorname{Aut}(\mathfrak{g}) \times \mathscr{L} \longrightarrow \mathscr{L}, \quad (\theta, g) \longmapsto \theta^{\ast}g = g_{\theta}$, where $g_{\theta}(u, v) = g(\theta^{-1}u, \theta^{-1}v) \,\, \forall u, v\in \mathfrak{g}$. See \cite{ha2012isometry} for more detail. The isotropy subgroup of the metric $g$ under the above action is given by $\operatorname{Aut}(\mathfrak{g})_{g} = \left\lbrace \theta \in \operatorname{Aut}(\mathfrak{g}) / \theta^{\ast}g = g \right\rbrace.$
\begin{definition}
The group $\operatorname{Aut}(G)_g \cong \operatorname{Aut}(\mathfrak{g})_g$ is called the group of isometric automorphisms of $G$.
\end{definition}
\begin{definition}
A Lie group $G$ is said to be of type $(R)$ if, for every element $x$ in the Lie algebra $\mathfrak{g}$ of $G$, the endomorphism $\operatorname{ad} x : \mathfrak{g} \to \mathfrak{g}$, defined by $y \mapsto [x, y]$, has only real eigenvalues. Some authors refer to this condition by saying that $\mathfrak{g}$ has only real roots.
\end{definition}
The main key to solving our problem is given by the following three lemmas and theorem.
\begin{lemma} \label{normal} (Lemma 2.2 in \cite{ha2012isometry}). The group $L(G)$ of left translations on $G$ is a normal subgroup of $\operatorname{Isom}(G, g)$ if and only if $\operatorname{Isom}(G, g)_e = \operatorname{Aut}(G)_g$.
\end{lemma}
\begin{lemma} \label{gordon} (Corollary (5.3) in \cite{gordon1988isometry}).
Let $G$ be a simply connected solvable Lie group whose Lie algebra has only real roots. Let $g$ and $g'$ be two left-invariant metrics on $G$. Then $g$ is isometric to $g'$ if and only if $g' = \psi^{\ast} g$ for some $\psi \in \operatorname{Aut}(G)$.
\end{lemma}
\begin{lemma} \label{helgason} (Lemma 11.2. page 62 in \cite{helgason}).
Let $M$ be a Riemannian manifold, $\phi$ and $\psi$ two isometries of $M$ onto itself. Suppose there exists a point $p \in M$ for which $\phi(p) = \psi(p)$ and $d_p\phi = d_p\psi$. Then $\phi = \psi$.
\end{lemma}
\begin{theorem}
Let $G$ be a simply connected solvable Lie group of type $(R)$, then for each left invariant Riemannian metric $g$ on $G$, $\operatorname{Isom}(G, g) \cong G \rtimes \operatorname{Aut}(G)_g$.
\end{theorem}
\begin{proof}
According to the lemma \ref{normal}, it suffices to show that $\operatorname{Isom}(G, g)_e = \operatorname{Aut}(G)_g$. Let $\theta \in \operatorname{Isom}(G, g)_e$, then $\theta(e) = e$ and $\theta^{\ast}g = g$, i.e. $g(u, v) = g(\theta_{\ast}^{-1}u, \theta_{\ast}^{-1}v) \,\, \forall u, v\in \mathfrak{g}$. On the other hand, since $G$ is a simply connected solvable Lie group of type $(R)$, we have by lemma \ref{gordon}, $\psi^{\ast}g = g$ for some $\psi \in \operatorname{Aut}(G)$, i.e. $g(u, v) = g(\psi_{\ast}^{-1}u, \psi_{\ast}^{-1}v) \,\, \forall u, v\in \mathfrak{g}$.
Hence $\theta_{\ast} = d_e\theta = \psi_{\ast} = d_e\psi$ and $\theta(e) = \psi(e) = e$. By lemma \ref{helgason}, we have:\\ $\theta = \psi \in \operatorname{Aut}(G)_g$. Therefore $\operatorname{Isom}(G, g) \cong G \rtimes \operatorname{Aut}(G)_g$. (See page 192 in \cite{ha2012isometry}).
\end{proof}
Below, we provide a comprehensive list of all 4-dimensional nonunimodular Lie algebras \cite{kremlev2010signature}
\begin{center}
\textbf{Table 1} \vspace{2mm}
	
\begin{tabular}{|p{3.5cm}|p{10cm}|} \hline
Lie algebra  & Nonzero commutators  \\ \hline
$ \operatorname{A}_2 \oplus \operatorname{2A}_1 $ & $[e_1, e_2] = e_2$ \\ \hline
$ \operatorname{2A}_2 $ &  $[e_1, e_2] = e_2, [e_3, e_4] = e_4$ \\ \hline
$ \operatorname{A}_{3, 2} \oplus \operatorname{A}_1 $ & $[e_1, e_3] = e_1, [e_2, e_3] = e_1 + e_2$ \\ \hline
$ \operatorname{A}_{3, 3} \oplus \operatorname{A}_1 $ & $[e_1, e_3] = e_1, [e_2, e_3] = e_2$ \\ \hline
$\operatorname{A}_{3, 5}^{\alpha} \oplus \operatorname{A}_1, \newline
0 < \left| \alpha \right|  < 1 $  & $[e_1, e_3] = e_1, [e_2, e_3] = \alpha e_2$ \\ \hline
$ \operatorname{A}_{3, 7}^{\alpha} \oplus \operatorname{A}_1, \alpha > 0 $ &$[e_1, e_3] = \alpha e_1 - e_2, [e_2, e_3] = e_1 + \alpha e_2$ \\ \hline
$ \operatorname{A}_{4, 2}^{\alpha}, \alpha \neq 0,\newline
\alpha \neq -2 $ & $[e_1, e_4] = \alpha e_1, [e_2, e_4] = e_2, [e_3, e_4] = e_2 + e_3$ \\ \hline
$ \operatorname{A}_{4, 3} $ & $[e_1, e_4] = e_1, [e_3, e_4] = e_2$ \\ \hline
$ \operatorname{A}_{4, 4} $ & $[e_1, e_4] = e_1, [e_2, e_4] = e_1 + e_2, [e_3, e_4] = e_2 + e_3$  \\ \hline
$ \operatorname{A}_{4, 5}^{\alpha, \beta}, \alpha\beta \neq 0, \newline
-1 \leq \alpha \leq \beta \leq 1,\newline
\alpha + \beta \neq -1 $ & $\newline
[e_1, e_4] = e_1, [e_2, e_4] = \alpha e_2, [e_3, e_4] = \beta e_3$  \\ \hline
$ \operatorname{A}_{4, 6}^{\alpha, \beta}, \alpha \neq 0, \newline
\beta \geq 0, \alpha \neq -2\beta $ & $[e_1, e_4] = \alpha e_1, [e_2, e_4] = \beta e_2 - e_3, [e_3, e_4] = e_2 + \beta e_3$  \\ \hline	
$ \operatorname{A}_{4, 7} $ & $[e_2, e_3] = e_1, [e_1, e_4] = 2e_1, [e_2, e_4] = e_2, [e_3, e_4] = e_2 + e_3$  \\ \hline
$ \operatorname{A}_{4, 9}^{\beta}, -1 < \beta \leq 1 $ & $[e_2, e_3] = e_1, [e_1, e_4] = (1 + \beta)e_1, [e_2, e_4] = e_2, \newline
[e_3, e_4] = \beta e_3$  \\ \hline
$ \operatorname{A}_{4, 11}^{\alpha}, \alpha > 0 $ & $[e_2, e_3] = e_1, [e_1, e_4] = 2\alpha e_1, [e_2, e_4] = \alpha e_2 - e_3, \newline
[e_3, e_4] = e_2 + \alpha e_3$  \\ \hline
$ \operatorname{A}_{4, 12} $ & $[e_1, e_3] = e_1, [e_2, e_3] = e_2, [e_1, e_4] = -e_2, [e_2, e_4] = e_1$  \\ \hline
\end{tabular}
\end{center}
\begin{remark}
All the Lie algebras in Table 1 are solvable. We will focus on those of type $(R)$ and use the preceding theorem to describe their isometry groups. In fact, as we will show in the following proposition, only four Lie algebras from Table 1 are not of type $(R)$.
\end{remark}
\begin{proposition}
All the Lie algebras from Table 1 are solvable of type $(R)$ except the following four Lie algebras: $\operatorname{A}_{3, 7}^{\alpha} \oplus \operatorname{A}_1, \operatorname{A}_{4, 6}^{\alpha, \beta}, \operatorname{A}_{4, 11}^{\alpha}$ and $\operatorname{A}_{4, 12}$.
\end{proposition}
\begin{proof}
For the Lie algebra $\operatorname{A}_{3, 7}^{\alpha} \oplus \operatorname{A}_1$, the endomorphism $\operatorname{ad} e_3$ is represented by the matrix
$$\operatorname{ad} e_3 = \begin{bmatrix}
-\alpha & -1 & 0 & 0\\
1 & -\alpha & 0 & 0\\
0 & 0 & 0 & 0\\
0 & 0 & 0 & 0
\end{bmatrix}.$$
We see that $\operatorname{ad} e_3$ has two complex eigenvalues which are $-\alpha + i$ and $-\alpha - i$.\\
For the Lie algebra $\operatorname{A}_{4, 6}^{\alpha, \beta}$, the endomorphism $\operatorname{ad} e_4$ is represented by the matrix
$$\operatorname{ad} e_4 = \begin{bmatrix}
-\alpha & 0 & 0 & 0\\
0 & -\beta & -1 & 0\\
0 & 1 & -\beta & 0\\
0 & 0 & 0 & 0
\end{bmatrix}.$$
We see that $\operatorname{ad} e_4$ has two complex eigenvalues which are $-\beta + i$ and $-\beta - i$.\\
For the Lie algebra $\operatorname{A}_{4, 11}^{\alpha}$, the endomorphism $\operatorname{ad} e_4$ is represented by the matrix
$$\operatorname{ad} e_4 = \begin{bmatrix}
-2\alpha & 0 & 0 & 0\\
0 & -\alpha & -1 & 0\\
0 & 1 & -\alpha & 0\\
0 & 0 & 0 & 0
\end{bmatrix}.$$
Thus, $\operatorname{ad} e_4$ has two complex eigenvalues which are $-\alpha + i$ and $-\alpha - i$.\\
For the Lie algebra $\operatorname{A}_{4, 12}$, the endomorphism $\operatorname{ad} e_4$ is represented by the matrix
$$\operatorname{ad} e_4 = \begin{bmatrix}
0 & -1 & 0 & 0\\
1 & 0 & 0 & 0\\
0 & 0 & 0 & 0\\
0 & 0 & 0 & 0
\end{bmatrix}.$$
Therefore, the element $\operatorname{ad} e_4$ has two complex eigenvalues which are $\pm i$.\\
For all other Lie algebras in Table 1, one can easily verify that $\operatorname{ad} e_i$, for all $i = 1, \dots, 4$, have only real eigenvalues.
\end{proof}
\section{The simply connected Lie groups of dimension four}
In this section, we introduce the simply connected Lie groups associated with the four-dimensional nonunimodular solvable Lie algebras of type $(R)$. This problem is addressed by Biggs and Remsing in \cite{biggs2016classification}. We adopt their notation and provide some additional explanations.
\begin{center}
	\newpage
	
\textbf{Table 2} \vspace{2mm}
	
\begin{tabular}{|p{3.5cm}|p{3.5cm}|p{3.5cm}|} \hline
Lie algebra  &  Notation in \cite{biggs2016classification} & Simply connected Lie group \\ \hline
$ \operatorname{A}_{2} \oplus 2\operatorname{A}_1 $ & $ \mathfrak{g}_{2.1} \oplus 2\mathfrak{g}_1 $ & $G_{2.1} \times \mathbb{R}^2$ \\ \hline
$ \operatorname{2A}_2 $ &  not defined & $G_2$ \\ \hline
$ \operatorname{A}_{3, 2} \oplus \operatorname{A}_1 $ & $ \mathfrak{g}_{3.2} \oplus \mathfrak{g}_1 $ & $G_{3.2} \times \mathbb{R}$ \\ \hline
$ \operatorname{A}_{3, 3} \oplus \operatorname{A}_1 $ & $ \mathfrak{g}_{3.3} \oplus \mathfrak{g}_1 $ & $G_{3.3} \times \mathbb{R}$ \\ \hline
$\operatorname{A}_{3, 5}^{\alpha} \oplus \operatorname{A}_1$  & not defined & $G_{c} \times \mathbb{R}$ \\ \hline
$ \operatorname{A}_{4, 2}^{\alpha}, \alpha \neq (0, 1)$ & $ \mathfrak{g}_{4.2}^{\alpha}$ & $G_{4.2}^{\alpha}$ \\ \hline
$ \operatorname{A}_{4, 2}^{1}$ & $ \mathfrak{g}_{4.2}^{1}$ & $G_{4.2}^{1}$ \\ \hline
$ \operatorname{A}_{4, 3} $ & $ \mathfrak{g}_{4.3} $ & $G_{4.3}$ \\ \hline
$ \operatorname{A}_{4, 4} $ & $ \mathfrak{g}_{4.4} $ &  $G_{4.4}$ \\ \hline
$ \operatorname{A}_{4, 5}^{\alpha, \beta}, \alpha\beta \neq 0$ & $ \mathfrak{g}_{4.5}^{\alpha, \beta}$ & $G_{4.5}^{\alpha, \beta}$  \\ \hline	
$ \operatorname{A}_{4, 7} $ & $ \mathfrak{g}_{4.7} $ & $G_{4.7}$  \\ \hline
$ \operatorname{A}_{4, 9}^{\beta} $ & $ \mathfrak{g}_{4.8}^{\alpha} $ & $G_{4.8}^{\alpha} $ \\ \hline
\end{tabular}
\end{center}
\begin{remark}
The simply connected Lie groups associated with these Lie algebras will be described in details in the following subsections.
\end{remark}
\subsection{The simply connected Lie group associated with $\operatorname{A}_{2} \oplus 2\operatorname{A}_1$}
The Lie algebra $\operatorname{A}_{2} \oplus 2\operatorname{A}_1$ has a basis $\mathcal{B} = \left\lbrace e_1, e_2, e_3, e_4\right\rbrace$ such that the only nonzero bracket is $[e_1, e_2] = e_2$. Hence we see that this Lie algebra is a direct sum of the two-dimensional affine Lie algebra $\mathfrak{aff}(\mathbb{R})$ with the abelian Lie algebra $\mathbb{R}^2$. The affine Lie group $\operatorname{Aff}(\mathbb{R})$ is isomorphic to the semidirect product $\mathbb{R} \rtimes \mathbb{R}^{\times}$ where $\mathbb{R}^{\times}$ acts on $\mathbb{R}$ by multiplication. The Lie groups $\mathbb{R} \rtimes \mathbb{R}^{\times}$ and $\mathbb{R} \rtimes \mathbb{R}_{+}^{\times}$ have the same Lie algebra wich is $\mathfrak{aff}(\mathbb{R})$, in addition $\mathbb{R}_{+}^{\times}$ is simply connected and isomorphic to $\mathbb{R}$ by the exponential: $\exp: \mathbb{R} \longrightarrow \mathbb{R}_{+}^{\times}, \; t \longmapsto e^t$. Thus, the simply connected Lie group associated with the Lie algebra $\mathfrak{aff}(\mathbb{R})$ is
$$G_{2.1} := \widetilde{\operatorname{Aff}(\mathbb{R})} = \left\lbrace \begin{bmatrix}
e^a & b\\
0 & 1
\end{bmatrix} \big/ a, b \in \mathbb{R}\right\rbrace.$$
Consequently, the simply connected Lie group with Lie algebra $\operatorname{A}_{2} \oplus 2\operatorname{A}_1$ is $G_{2.1} \times \mathbb{R}^2$.
\subsection{The simply connected Lie group associated with $\operatorname{2A_2}$}
The Lie algebra $\operatorname{2A}_2$ has a basis $\mathcal{B} = \left\lbrace e_1, e_2, e_3, e_4\right\rbrace$ such that the only nonzero brackets are 
$$[e_1, e_2] = e_2, \qquad [e_3, e_4] = e_4.$$
Hence we see that this Lie algebra is a direct sum of the two-dimensional affine Lie algebra $\mathfrak{aff}(\mathbb{R})$ with itself, i.e. $\operatorname{2A_2} = \mathfrak{aff}(\mathbb{R}) \oplus \mathfrak{aff}(\mathbb{R})$. Therefore, the simply connected Lie group associated with our Lie algebra $\operatorname{2A_2} = \mathfrak{aff}(\mathbb{R}) \oplus \mathfrak{aff}(\mathbb{R})$ is the following
$$G_2 := \widetilde{\operatorname{Aff}(\mathbb{R})} \times \widetilde{\operatorname{Aff}(\mathbb{R})} = \left\lbrace \begin{bmatrix}
e^a & b & 0 & 0\\
0 & 1 & 0 & 0\\
0 & 0 & e^c & d\\
0 & 0 & 0 & 1
\end{bmatrix} \bigg/ a, b, c, d \in \mathbb{R}\right\rbrace.$$
\subsection{The simply connected Lie group associated with $\operatorname{A}_{3, 2} \oplus \operatorname{A}_1$}
The Lie algebra $\operatorname{A}_{3, 2} \oplus \operatorname{A}_1$ is the Lie algebra $\mathfrak{g}_{3.2} \oplus \mathfrak{g}_1$ in \cite{biggs2016classification}. Its associated simply connected Lie group is $G_{3.2} \times \mathbb{R}$ where
$$G_{3.2} = \left\lbrace \begin{bmatrix}
1 & 0 & 0\\
y & e^z & 0\\
x & -ze^z & e^z
\end{bmatrix} \;:\; x, y, z \in \mathbb{R}\right\rbrace \quad (\text{See page 1005 in \cite{biggs2016classification}}).$$
\subsection{The simply connected Lie group associated with $\operatorname{A}_{3, 3} \oplus \operatorname{A}_1$}
The Lie algebra $\operatorname{A}_{3, 3} \oplus \operatorname{A}_1$ is the Lie algebra $\mathfrak{g}_{3.3} \oplus \mathfrak{g}_1$ in \cite{biggs2016classification}. Its associated simply connected Lie group is $G_{3.3} \times \mathbb{R}$ where
$$G_{3.3} = \left\lbrace \begin{bmatrix}
1 & 0 & 0\\
y & e^z & 0\\
x & 0 & e^z
\end{bmatrix} \;:\; x, y, z \in \mathbb{R}\right\rbrace \quad (\text{See page 1005 in \cite{biggs2016classification}}).$$
\subsection{The simply connected Lie group associated with $\operatorname{A}_{3, 5}^{\alpha} \oplus \operatorname{A}_1$}
The Lie algebra $\operatorname{A}_{3, 5}^{\alpha} \oplus \operatorname{A}_1$ has a basis $\mathcal{B} = \left\lbrace e_1, e_2, e_3, e_4\right\rbrace$ such that the nonzero brackets are 
$$[e_1, e_3] = e_1,\quad [e_2, e_3] = \alpha e_2,\quad \text{where}\; 0 < |\alpha| < 1.$$
Hence $\operatorname{A}_{3, 5}^{\alpha} = \langle e_1, e_2, e_3\rangle$ is a three-dimensional solvable Lie algebra that is nonunimodular. In the other hand, $\operatorname{A}_1 = \langle e_4\rangle \cong \mathbb{R}$ is a one-dimensional Lie algebra.
\begin{proposition}
We can realize the Lie algebra $\operatorname{A}_{3, 5}^{\alpha}$ as a semidirect sum $\operatorname{A}_{3, 5}^{\alpha} \cong \mathbb{R}^2 \rtimes_{\sigma_{\alpha}}\mathbb{R}$ where $\sigma_{\alpha}(t) = \begin{bmatrix}
-t & 0\\
0 & -t\alpha
\end{bmatrix}$.
\end{proposition}
\begin{proof}
Consider the following basis $\mathscr{B} = \left\lbrace X, Y, Z\right\rbrace$ of $\mathbb{R}^2 \rtimes_{\sigma_{\alpha}}\mathbb{R}$ where
$$X = \left(\begin{bmatrix}
1\\
0
\end{bmatrix}, 0 \right), \quad Y = \left(\begin{bmatrix}
0\\
1
\end{bmatrix}, 0 \right), \quad Z = \left(\begin{bmatrix}
0\\
0
\end{bmatrix}, 1 \right).$$
We recall that if $(x, t), (y, r) \in \mathbb{R}^2 \rtimes_{\sigma_{\alpha}}\mathbb{R}$, where $x, y \in \mathbb{R}^2$ and $t, r \in \mathbb{R}$. Then their bracket is given by
$$[(x, t), (y, r)] = \big([x, y]_{\mathbb{R}^2} + \sigma_{\alpha}(t)(y) - \sigma_{\alpha}(r)(x), [t, r]_{\mathbb{R}}\big)$$
where $[, ]_{\mathbb{R}^2}$ and $[, ]_{\mathbb{R}}$ are the brackets in $\mathbb{R}^2$ and $\mathbb{R}$ respectively. According to this fact, we obtain that the brackets of elements of $\mathscr{B}$ are
$$[X, Y] = 0, \qquad [X, Z] = X, \qquad [Y, Z] = \alpha Y.$$
These are exactly the brackets of the Lie algebra $\operatorname{A}_{3, 5}^{\alpha}$. Hence $\operatorname{A}_{3, 5}^{\alpha} \cong \mathbb{R}^2 \rtimes_{\sigma_{\alpha}}\mathbb{R}$.
\end{proof}
According to \cite{ha2009left}, any three-dimensional solvable nonunimodular Lie algebra is isomorphic to either $\mathfrak{g}_I$ or $\mathfrak{g}_c$ for some $c \in \mathbb{R}$ where
$$\mathfrak{g}_I \cong \mathbb{R}^2 \rtimes_{\sigma_{I}}\mathbb{R}, \quad\text{where}\quad \sigma_{I}(t) = \begin{bmatrix}
t & 0\\
0 & t
\end{bmatrix};$$
$$\mathfrak{g}_c \cong \mathbb{R}^2 \rtimes_{\sigma_{c}}\mathbb{R}, \quad\text{where}\quad \sigma_{c}(t) = \begin{bmatrix}
0 & -ct\\
t & 2t
\end{bmatrix}.$$
\begin{proposition}
The Lie algebra $\operatorname{A}_{3, 5}^{\alpha}$ is isomorphic to $\mathfrak{g}_c$ for some $c \in \mathbb{R}$, in fact we have
$$\operatorname{A}_{3, 5}^{\alpha} \cong \mathbb{R}^2 \rtimes_{\sigma_{\alpha}}\mathbb{R} \cong \mathbb{R}^2 \rtimes_{\sigma_{c}}\mathbb{R} \cong \mathfrak{g}_c.$$
\end{proposition}
\begin{proof}
It is clear that the Lie algebra $\mathbb{R}^2 \rtimes_{\sigma_{\alpha}}\mathbb{R}$ is not isomorphic to $\mathbb{R}^2 \rtimes_{\sigma_{I}}\mathbb{R}$; Because if they are isomorphic, then there exists a matrix $P \in \operatorname{GL}(2, \mathbb{R}) = \operatorname{Aut}(\mathbb{R}^2)$ \edit{and a nonzero real number $\lambda$ such that $\sigma_{I}(\lambda t) = P\sigma_{\alpha}(t)P^{-1}$. This means that $\sigma_{I}(\lambda t)$ and $\sigma_{\alpha}(t)$ have the same eigenvalues, this implies that $t = 0$ or $\lambda = -1$ and $\alpha = 1$. Both cases are impossible because $\operatorname{A}_{3, 5}^{\alpha}$ is not abelian and $0 < |\alpha| < 1$. Therefore $\operatorname{A}_{3, 5}^{\alpha}$ is isomorphic to $\mathfrak{g}_c$ for some $c \in \mathbb{R}$. In fact, put $\lambda = -\frac{1 + \alpha}{2}$ and $c = \frac{4\alpha}{(1 + \alpha)^2}$. Then the matrix $\sigma_{c}(\lambda t)$ has eigenvalues $-t$ and $-\alpha t$. Consequently $\sigma_{c}(\lambda t)$ is similar to $\sigma_{\alpha}(t)$ and hence $\operatorname{A}_{3, 5}^{\alpha}$ and $\mathfrak{g}_c$ are isomorphic.}
\end{proof}
The simply connected Lie group associated with the Lie algebra $\mathfrak{g}_c$ is $G_c = \mathbb{R}^2 \rtimes_{\varphi_{c}}\mathbb{R}$. See pages 873 and 874 in \cite{ha2009left}.\\
We conclude that the simply connected Lie group associated with the Lie algebra $\operatorname{A}_{3, 5}^{\alpha} \oplus \operatorname{A}_1$ is $G_c \times \mathbb{R}$.
\subsection{The simply connected Lie groups associated with nonunimodular indecomposable Lie algebras of dimension four}
The simply connected Lie groups associated with the indecomposable Lie algebras of type $(R)$ in Table 1 are all described by \cite{biggs2016classification}. We recall their structures as follows
\begin{eqnarray*}
G_{4.2}^{\alpha} &=& \left\lbrace \begin{bmatrix}
e^{-\alpha z} & 0 & 0 & w\\
0 & e^{-z} & -\alpha ze^{-z} & \alpha x\\
0 & 0 & e^{-z} & y\\
0 & 0 & 0 & 1
\end{bmatrix} \;:\; w, x, y, z \in \mathbb{R}\right\rbrace\\
G_{4.2}^{1} &=& \left\lbrace \begin{bmatrix}
e^{-z} & 0 & 0 & w\\
0 & e^{-z} & -ze^{-z} & x\\
0 & 0 & e^{-z} & y\\
0 & 0 & 0 & 1
\end{bmatrix} \;:\; w, x, y, z \in \mathbb{R}\right\rbrace\\
G_{4.3} &=& \left\lbrace \begin{bmatrix}
e^{-z} & 0 & 0 & w\\
0 & 1 & -z & x\\
0 & 0 & 1 & y\\
0 & 0 & 0 & 1
\end{bmatrix} = p(w, x, y, z) \;:\; w, x, y, z \in \mathbb{R}\right\rbrace\\
G_{4.4} &=& \left\lbrace \begin{bmatrix}
e^{-z} & -ze^{-z} & \frac{1}{2}z^2e^{-z} & w\\
0 & e^{-z} & -ze^{-z} & x\\
0 & 0 & e^{-z} & y\\
0 & 0 & 0 & 1
\end{bmatrix} \;:\; w, x, y, z \in \mathbb{R}\right\rbrace\\
G_{4.5}^{\alpha, \beta} &=& \left\lbrace \begin{bmatrix}
e^{-z} & 0 & 0 & w\\
0 & e^{-\alpha z} & 0 & y\\
0 & 0 & e^{-\beta z} & x\\
0 & 0 & 0 & 1
\end{bmatrix} \;:\; w, x, y, z \in \mathbb{R}\right\rbrace\\
G_{4.7} &=& \left\lbrace \begin{bmatrix}
e^{-2z} & -ye^{-z} & (x + yz)e^{-z} & 2w\\
0 & e^{-z} & -ze^{-z} & x\\
0 & 0 & e^{-z} & y\\
0 & 0 & 0 & 1
\end{bmatrix} \;:\; w, x, y, z \in \mathbb{R}\right\rbrace
\end{eqnarray*}
\begin{eqnarray*}
G_{4.8}^{\alpha} &=& \left\lbrace \begin{bmatrix}
e^{-(1 + \alpha)z} & x & w\\
0 & e^{-\alpha z} & y\\
0 & 0 & 1
\end{bmatrix} \;:\; w, x, y, z \in \mathbb{R}\right\rbrace.
\end{eqnarray*}
\section{Metrics on nonunimodular 4-dimensional Lie groups}
A left-invariant Riemannian metric on a Lie group $G$ is an inner product $\langle ., .\rangle$ defined on the tangent space of $G$, which varies smoothly as one moves across the group, and is preserved under left translations. In other words, the left translations are isometries with respect to this inner product. This means that 
$$\langle u, v\rangle_{p} = \langle d_{p} L_{a}(u), d_{p} L_{a}(v)\rangle_{ap} \quad \forall p, a \in G \quad \forall u, v \in T_{p}G.$$
This is the general definition of an isometry. However, because the isometry group of a left-invariant Riemannian metric on a Lie group $G$ is entirely determined by the isotropy subgroup at the identity element $e$, we can reformulate this definition in the form presented in the introduction.\\
There is a bijective correspodence between left invariant Riemannian metrics on a Lie group $G$, and inner products on the Lie algebra $\mathfrak{g}$ of $G$. In fact, if $\langle ., .\rangle$ is an inner product on $\mathfrak{g}$, put 
$$\langle u, v\rangle_{p} = \langle d_{p} L_{p^{-1}}(u), d_{p} L_{p^{-1}}(v)\rangle \quad \forall p \in G \quad \forall u, v \in T_{p}G.$$
Then $\langle ., .\rangle_{p}$ defines a left invariant Riemannian metric on $G$. Thus, the classification of left invariant Riemannian metrics on a Lie group $G$, is equivalent to the classification of inner products on its Lie algebra $\mathfrak{g}$.\\
The space of inner products on $\mathfrak{g}$ is the space of symmetric, positive definite matrices denoted $\operatorname{S}_n$. Consider the following notations\\
$\operatorname{Tsup}_n$ := the group consisting of upper triangular matrices with positive diagonal entries.\\
$\operatorname{Tinf}_n$ := the group consisting of lower triangular matrices with positive diagonal entries.\\
$\operatorname{D}_n^{+}$ := the group consisting of diagonal matrices with positive elements.
\begin{proposition} \label{bijection}
The following map is bijective
$$ \begin{array}{rcl}
\varphi : \operatorname{Tsup}_{n}&\longrightarrow& \operatorname{S}_{n}\\
B &\longmapsto& (B^{-1})^t(B^{-1})
\end{array} $$
\end{proposition}
\begin{proof}
It is clear that the map $\varphi$ is well defined. For the injectivity, let $B, C \in \operatorname{Tsup}_{n}$ such that $\varphi(B) = \varphi(C)$. This is equivalent to the following equality $B^{-1}C = (C^{-1}B)^{t}$. Put $M = C^{-1}B$, then 
$B^{-1}C = (C^{-1}B)^{t} \Leftrightarrow M^{-1} = M^{t}$. Since $\operatorname{Tsup}_{n}$ is a group, then $M, M^{-1} \in \operatorname{Tsup}_{n}$. Since $M^{-1} = M^{t} \in \operatorname{Tinf}_n$, then $M \in \operatorname{Tsup}_n \cap \operatorname{Tinf}_n = \operatorname{D}_n^{+}$. Hence we have $M^{-1} = M^{t} = M$, then $M = I_n$. This implies that $B = C$ and $\varphi$ is injective.\\
For the surjectivity, let $A \in \operatorname{S}_{n}$ and consider the inner product on $\mathbb{R}^{n}$ defined by 
$$\langle u, v\rangle = u^{t}Av, \forall u, v \in \mathbb{R}^{n}.$$
Let $\mathcal{B} = \left\lbrace e_1, ..., e_n\right\rbrace$ be the canonical basis of $\mathbb{R}^{n}$. The Gram-Schmidt procedure produces an orthonormal basis $\mathscr{B} = \left\lbrace v_1, ..., v_n\right\rbrace$ of the Euclidean space $\left( \mathbb{R}^{n}, \langle ., .\rangle\right)$ such that
$$[v_1, ..., v_n] = X \in Tsup_n.$$
This means that $Xe_i = v_i$. Since $\mathscr{B} = \{v_1, ..., v_n\}$ is an orthonormal basis of $(\mathbb{R}^n, \langle ., . \rangle)$, then
$$\operatorname{Mat}\left(\langle ., .\rangle , \mathscr{B}\right) =X^tAX = I_n.$$
Thus $A = (X^{-1})^{t}(X^{-1}) = \varphi(X)$. Hence $\varphi$ is surjective.\\
Therefore $\varphi$ is a bijection. This result was already proven in \cite{van2017metrics}.
\end{proof}
\begin{proposition} \cite{van2017metrics} 
The set of inner products on $\mathfrak{g}$ is $(n^2 + n)/2$-dimensional, and can be identified with the set of upper triangular matrices with positive diagonal entries.
\end{proposition}
The bijection $\varphi$ in Proposition \ref{bijection} is the key that simplifies the classification of left-invariant Riemannian metrics on Lie groups. It was used in \cite{van2017metrics,aitbenhaddou2025classification} to obtain a classification of left-invariant Riemannian metrics on four-dimensional Lie groups.
\section{The isometry groups}
Let $\theta$ be an automorphism of $G$ and let $g$ be a left invariant Riemannian metric on $G$. The pullback $\theta^{\ast}g$ of $g$ by $\theta$ is defined by:
$\theta^{\ast}g(u, v) = g\left(\theta_{\ast}^{-1}(u), \theta_{\ast}^{-1}(v)\right) \; \forall u, v \in \mathfrak{g}$ where $\theta_{\ast}$ is the differential of $\theta$ at the identity element $e$ of $G$. This equation is equivalent to the following matrix equation $[g] = [\theta_{\ast}]^t[\theta^{\ast}g][\theta_{\ast}]$. Hence in term of matrix calculation, the group of isometric automorphisms of $G$ is given by \cite{ha2012isometry}
$$\operatorname{Aut}(G)_{g} = \left\lbrace \theta \in \operatorname{Aut}(G)\,\, /\,\, [g] = [\theta_{\ast}]^t[g][\theta_{\ast}] \right\rbrace.$$
It is well known that if two metrics are equivalent up to a diffeomorphism (or, in particular, up to an automorphism), then their isometry groups are conjugate; see \cite{ayad2024isometry, ha2012isometry}. Thus, it suffices to describe the isometry group of the representative metrics for non-equivalent classes of metrics on Lie groups.
\begin{remark}
\begin{enumerate}
\item In each nonunimodular 4-dimensional Lie group, we will investigate different metrics that yield different isometry groups. This means that for any other metric, its isometry group is isomorphic to that of one of the investigated metrics.
\item In this paper, we choose the first class of metrics to be the one with the maximal group of isometric automorphisms.
\end{enumerate}
\end{remark}
\begin{remark} \label{orthogonal}
The isotropy subgroup $\operatorname{Isom}(G, g)_e$ is compact. If $\theta$ is an element of $\operatorname{Isom}(G, g)_e$, the differential $\theta_{\ast} = d_e\theta : \mathfrak{g} \longrightarrow \mathfrak{g}$ of $\theta$ at $e$ is an orthogonal transformation on the vector space $\mathfrak{g} = T_eG$. That is, for any $x, y \in T_eG$: $g(\theta_{\ast}^{-1}x, \theta_{\ast}^{-1}y) = g(x, y).$\\
Therefore, by lemma \ref{helgason}, $\operatorname{Isom}(G, g)_e$ can be injected in $\operatorname{O}(\mathfrak{g}, g)$. Hence
$$\operatorname{Aut}(G)_g \subset \operatorname{Isom}(G, g)_e \subset \operatorname{O}(4), \quad (\text{see page 193 in \cite{ha2012isometry}}).$$
Throughout this paper, $\operatorname{O}(n)$ denotes the orthogonal group in $n$ dimensions.
\end{remark}
\subsection{The isometry group of the Lie group $G_{2.1} \times \mathbb{R}^2$}
By theorem 3.1. in \cite{aitbenhaddou2025classification}, any left-invariant Riemannian metric on $\operatorname{A}_{2} \oplus 2\operatorname{A}_1 = \operatorname{Lie}\left(G_{2.1} \times \mathbb{R}^2\right)$ is equivalent, up to automorphism, to the following metric
$$
U = \begin{bmatrix}
\alpha & \beta & \gamma & 0\\
0 & 1 & \lambda & \edit{\delta}\\
0 & 0 & 1 & 0\\
0 & 0 & 0 & 1
\end{bmatrix} \quad \alpha > 0, \quad \edit{\beta, \gamma \geq 0, \quad \lambda, \delta \in \mathbb{R}}.$$
We distinguish \edit{five} cases associated with this metric
\begin{enumerate}
\item If $\beta = \gamma = \lambda = \edit{\delta} = 0$ in $U$, then according to the bijection $\varphi$, the symmetric positive definite matrix associated to $U$ is $M_1 = (U^{-1})^t(U^{-1}) = \operatorname{diag}\left\lbrace \frac{1}{\alpha^2}, 1, 1, 1\right\rbrace$.
\item If $\beta \neq 0$ and $\gamma = \lambda = \edit{\delta} = 0$ in $U$, by the bijection $\varphi$, the symmetric positive definite matrix associated to $U$ is 
$$M_2 = (U^{-1})^t(U^{-1}) = \begin{bmatrix}
\frac{1}{\alpha^2} & \frac{-\beta}{\alpha^2} & 0 & 0\\
\frac{-\beta}{\alpha^2} & 1 + \frac{\beta^2}{\alpha^2} & 0 & 0\\
0 & 0 & 1 & 0\\
0 & 0 & 0 & 1
\end{bmatrix}.$$
\item If $\lambda \neq 0$ and $\beta = \gamma = \edit{\delta} = 0$ in $U$, by the bijection $\varphi$, the symmetric positive definite matrix associated to $U$ is 
$$M_3 = (U^{-1})^t(U^{-1}) = \begin{bmatrix}
\frac{1}{\alpha^2} & 0 & 0 & 0\\
0 & 1 & -\lambda & 0\\
0 & -\lambda & 1 + \lambda^2 & 0\\
0 & 0 & 0 & 1
\end{bmatrix}.$$
If we choose the metric $U$ where $\gamma > 0$ and $\beta = \lambda = \edit{\delta} = 0$, we obtain that its group of isometric automorphisms is similar to that of $M_3$.
\item If $\beta \neq 0$, $\lambda > 0$ and $\gamma = \edit{\delta} = 0$ in $U$, then based on the bijection $\varphi$, the symmetric positive definite matrix associated to $U$ is 
$$M_4 = (U^{-1})^t(U^{-1}) = \begin{bmatrix}
\frac{1}{\alpha^2} & \frac{-\beta}{\alpha^2} & \frac{\beta\lambda}{\alpha^2} & 0\\
\\
\frac{-\beta}{\alpha^2} & 1 + \frac{\beta^2}{\alpha^2} & \frac{-\beta^2\lambda}{\alpha^2} - \lambda & 0\\
\\
\frac{\beta\lambda}{\alpha^2} & \frac{-\beta^2\lambda}{\alpha^2} - \lambda & 1 + \lambda^2 + \frac{\beta^2\lambda^2}{\alpha^2} & 0\\
\\
0 & 0 & 0 & 1
\end{bmatrix}.$$
If we choose the metric $U$ where $\beta, \gamma > 0$, $\lambda \neq 0$ and $\edit{\delta = 0}$, we obtain that its group of isometric automorphisms is the same as that of the metric $M_4$.
\item \edit{If all the parameters in $U$ are nonzero, then based on the bijection $\varphi$, the symmetric positive definite matrix associated to $U$ is 
$$M_5 = (U^{-1})^t(U^{-1}) = \begin{bmatrix}
\frac{1}{\alpha^2} & -\frac{\beta}{\alpha^2} & \frac{\beta\lambda - \gamma}{\alpha^2} & \frac{\beta\delta}{\alpha^2} \\
\\
-\frac{\beta}{\alpha^2} & \frac{\beta^2}{\alpha^2} + 1 & \frac{\beta(-\beta\lambda + \gamma)}{\alpha^2} - \lambda & -\frac{\beta^2\delta}{\alpha^2} - \delta \\
\\
\frac{\beta\lambda - \gamma}{\alpha^2} & \frac{\beta(-\beta\lambda + \gamma)}{\alpha^2} - \lambda & \frac{(-\beta\lambda + \gamma)^2}{\alpha^2} + \lambda^2 + 1 & -\frac{(-\beta\lambda + \gamma)\beta\delta}{\alpha^2} + \lambda\delta \\
\\
\frac{\beta\delta}{\alpha^2} & -\frac{\beta^2\delta}{\alpha^2} - \delta & -\frac{(-\beta\lambda + \gamma)\beta\delta}{\alpha^2} + \lambda\delta & \frac{\beta^2\delta^2}{\alpha^2} + \delta^2 + 1
\end{bmatrix}.$$}
\end{enumerate}
\begin{theorem}
The group of isometric automorphisms of $G_{2.1} \times \mathbb{R}^2$ is given by
\begin{eqnarray*}
\operatorname{Aut}\left(G_{2.1} \times \mathbb{R}^2\right)_{M_1}  &=& \operatorname{diag}\left\lbrace 1, \pm1, \operatorname{O}(2)\right\rbrace \\
\operatorname{Aut}\left(G_{2.1} \times \mathbb{R}^2\right)_{M_2}  &\cong& \operatorname{O}(2)\\
\operatorname{Aut}\left(G_{2.1} \times \mathbb{R}^2\right)_{M_3}  &\cong& (\mathbb{Z}_2)^2\\
\operatorname{Aut}\left(G_{2.1} \times \mathbb{R}^2\right)_{M_4}  &\cong& \mathbb{Z}_2\\
\edit{\operatorname{Aut}\left(G_{2.1} \times \mathbb{R}^2\right)_{M_5}}  &=& \edit{\left\lbrace I_4\right\rbrace.}
\end{eqnarray*}
\end{theorem}
\begin{proof}
The automorphism group of $\operatorname{A}_{2} \oplus 2\operatorname{A}_1 = \operatorname{Lie}\left(G_{2.1} \times \mathbb{R}^2\right)$ consists of elements of the form \cite{christodoulakis2003automorphisms}
$$A = \begin{bmatrix}
1 & 0 & 0 & 0\\
a_5 & a_6 & 0 & 0\\
a_9 & 0 & a_{11} & a_{12}\\
a_{13} & 0 & a_{15} & a_{16}
\end{bmatrix}.$$
We first note that, by Remark \ref{orthogonal}, for any left-invariant Riemannian metric on the Lie group $G$, we have $\operatorname{Aut}(G)_g \subset \operatorname{O}(4)$. This implies that $a_5 = a_9 = a_{13} = 0$ in $A$. Therefore, it suffices to consider automorphisms of the form $A = \begin{bmatrix}
1 & 0 & 0 & 0\\
0 & a_6 & 0 & 0\\
0 & 0 & a_{11} & a_{12}\\
0 & 0 & a_{15} & a_{16}
\end{bmatrix}$. Let $A$ be an automorphism of this form, then
\begin{eqnarray*}
A \in \operatorname{Aut}\left(G_{2.1} \times \mathbb{R}^2\right)_{M_1} &\Leftrightarrow& A^tM_1A = M_1\\
&\Leftrightarrow& a_6 = \pm1, \begin{bmatrix}
a_{11} & a_{12}\\
a_{15} & a_{16}
\end{bmatrix} \in \operatorname{O}(2).
\end{eqnarray*}	Hence $\operatorname{Aut}\left(G_{2.1} \times \mathbb{R}^2\right)_{M_1}  = \operatorname{diag}\left\lbrace 1, \pm1, \operatorname{O}(2)\right\rbrace$.\\
For the metric $M_2$, we have
\begin{eqnarray*}
A \in \operatorname{Aut}\left(G_{2.1} \times \mathbb{R}^2\right)_{M_2} &\Leftrightarrow& A^tM_2A = M_2\\
&\Leftrightarrow& a_6 = 1, \begin{bmatrix}
a_{11} & a_{12}\\
a_{15} & a_{16}
\end{bmatrix} \in \operatorname{O}(2).
\end{eqnarray*}	Hence $\operatorname{Aut}\left(G_{2.1} \times \mathbb{R}^2\right)_{M_2}  = \operatorname{diag}\left\lbrace 1, 1, \operatorname{O}(2)\right\rbrace \cong \operatorname{O}(2)$.\\
For the metric $M_3$, one can see that
\begin{eqnarray*}
A \in \operatorname{Aut}\left(G_{2.1} \times \mathbb{R}^2\right)_{M_3} &\Leftrightarrow& A^tM_3A = M_3\\
&\Leftrightarrow& a_6 = a_{11} = \pm1, a_{12} = a_{15} = 0, a_{16} = \pm1
\end{eqnarray*}
Thus
\begin{eqnarray*}
\operatorname{Aut}\left(G_{2.1} \times \mathbb{R}^2\right)_{M_3} &=& \left\lbrace I_4, \begin{bmatrix}
1 & 0 & 0 & 0\\
0 & 1 & 0 & 0\\
0 & 0 & 1 & 0\\
0 & 0 & 0 & -1
\end{bmatrix}, \begin{bmatrix}
1 & 0 & 0 & 0\\
0 & -1 & 0 & 0\\
0 & 0 & -1 & 0\\
0 & 0 & 0 & 1
\end{bmatrix}, \begin{bmatrix}
1 & 0 & 0 & 0\\
0 & -1 & 0 & 0\\
0 & 0 & -1 & 0\\
0 & 0 & 0 & -1
\end{bmatrix} \right\rbrace\\
&\cong& (\mathbb{Z}_2)^2.
\end{eqnarray*}
For the metric $M_4$, one gets that
\begin{eqnarray*}
A \in \operatorname{Aut}\left(G_{2.1} \times \mathbb{R}^2\right)_{M_4} &\Leftrightarrow& A^tM_4A = M_4\\
&\Leftrightarrow& a_{12} = a_{15} = 0, a_6 = a_{11} = 1, a_{16} = \pm1
\end{eqnarray*}
Thus
$$\operatorname{Aut}\left(G_{2.1} \times \mathbb{R}^2\right)_{M_4} = \left\lbrace I_4, \operatorname{diag}\left\lbrace 1, 1, 1, -1\right\rbrace \right\rbrace \cong \mathbb{Z}_2.$$
\edit{Finally, it is easy to verify that the automorphism $\operatorname{diag}\left\lbrace 1, 1, 1, -1\right\rbrace$ does not preserve $M_5$. Hence $\operatorname{Aut}\left(G_{2.1} \times \mathbb{R}^2\right)_{M_5} = \left\lbrace I_4\right\rbrace$.}
\end{proof}
\begin{corollary}
The isometry group of $G_{2.1} \times \mathbb{R}^2$ is given by
\begin{eqnarray*}
\operatorname{Isom}\left(G_{2.1} \times \mathbb{R}^2, M_1\right)  &\cong& \left( G_{2.1} \times \mathbb{R}^2\right) \rtimes \operatorname{diag}\left\lbrace 1, \pm1, \operatorname{O}(2)\right\rbrace\\
\operatorname{Isom}\left(G_{2.1} \times \mathbb{R}^2, M_2\right)  &\cong& \left( G_{2.1} \times \mathbb{R}^2\right) \rtimes \operatorname{O}(2) \\
\operatorname{Isom}\left(G_{2.1} \times \mathbb{R}^2, M_3\right)  &\cong& \left( G_{2.1} \times \mathbb{R}^2\right) \rtimes (\mathbb{Z}_2)^2\\
\operatorname{Isom}\left(G_{2.1} \times \mathbb{R}^2, M_4\right)  &\cong& \left( G_{2.1} \times \mathbb{R}^2\right) \rtimes \mathbb{Z}_2\\
\edit{\operatorname{Isom}\left(G_{2.1} \times \mathbb{R}^2, M_5\right)}  &\cong& \edit{\left( G_{2.1} \times \mathbb{R}^2\right).}
\end{eqnarray*}
\end{corollary}
\subsection{The isometry group of the Lie group $G_2$}
By theorem 3.2. in \cite{aitbenhaddou2025classification}, any left-invariant Riemannian metric on $\operatorname{2A_2} = \operatorname{Lie}(G_2)$ is equivalent, up to automorphism, to the following metric
$$U = \begin{bmatrix}
\alpha & \beta & \gamma & \nu\\
0 & 1 & \lambda & \delta\\
0 & 0 & \mu & 0\\
0 & 0 & 0 & 1
\end{bmatrix} \quad \alpha, \mu > 0, \quad \beta, \nu \geq 0, \quad \gamma, \lambda, \delta \in \mathbb{R}.$$
We distinguish four cases associated with this metric
\begin{enumerate}
\item If $\beta = \nu = \gamma = \lambda = \delta = 0$ and $\alpha = \mu$, then according to the bijection $\varphi$, the symmetric positive definite matrix associated to $U$ is\\ $M_1 = (U^{-1})^t(U^{-1}) = \operatorname{diag}\left\lbrace \frac{1}{\alpha^2}, 1, \frac{1}{\alpha^2}, 1\right\rbrace$.
\item If $\beta = \nu = \gamma = \lambda = \delta = 0$ and $\alpha \neq \mu$, then based on the bijection $\varphi$, the symmetric positive definite matrix associated to $U$ is\\
$M_2 = (U^{-1})^t(U^{-1}) = \operatorname{diag}\left\lbrace \frac{1}{\alpha^2}, 1, \frac{1}{\mu^2}, 1\right\rbrace$.
\item If $\beta > 0$ and $\nu = \gamma = \lambda = \delta = 0$, then by the bijection $\varphi$, the symmetric positive definite matrix associated to $U$ is 
$$M_3 = (U^{-1})^t(U^{-1}) = \begin{bmatrix}
\frac{1}{\alpha^2} & \frac{-\beta}{\alpha^2} & 0 & 0\\
\frac{-\beta}{\alpha^2} & 1 + \frac{\beta^2}{\alpha^2} & 0 & 0\\
0 & 0 & \frac{1}{\mu^2} & 0\\
0 & 0 & 0 & 1
\end{bmatrix}.$$
\item If $\beta > 0$, $\nu \neq 0$ and $\gamma = \lambda = \delta = 0$, then according to the bijection $\varphi$, the symmetric positive definite matrix associated to $U$ is 
$$M_4 = (U^{-1})^t(U^{-1}) = \begin{bmatrix}
\frac{1}{\alpha^2} & \frac{-\beta}{\alpha^2} & 0 & \frac{-\nu}{\alpha^2}\\
\frac{-\beta}{\alpha^2} & 1 + \frac{\beta^2}{\alpha^2} & 0 & \frac{\beta\nu}{\alpha^2}\\
0 & 0 & \frac{1}{\mu^2} & 0\\
\frac{-\nu}{\alpha^2} & \frac{\beta\nu}{\alpha^2} & 0 & 1 + \frac{\nu^2}{\alpha^2}
\end{bmatrix}.$$
\end{enumerate}
We denote by $\mathbf{D}(4)$ the dihedral group of order $8$.
\begin{theorem} \label{possiblegroups}
The possible groups of isometric automorphisms of $G_2$ are
\begin{eqnarray*}
\operatorname{Aut}\left(G_2\right)_{M_1}  &\cong& \mathbf{D}(4)\\
\operatorname{Aut}\left(G_2\right)_{M_2}  &\cong& \left(\mathbb{Z}_2\right)^2\\
\operatorname{Aut}\left(G_2\right)_{M_3}  &\cong& \mathbb{Z}_2\\
\operatorname{Aut}\left(G_2\right)_{M_4}  &=& \{I_4\}.
\end{eqnarray*}
\end{theorem}
\begin{proof}
The automorphism group of $\operatorname{2A_2} = \operatorname{Lie}(G_2)$ consists of elements of the form \cite{christodoulakis2003automorphisms}
$$\begin{bmatrix}
1 & 0 & 0 & 0\\
a_5 & a_6 & 0 & 0\\
0 & 0 & 1 & a_{12}\\
0 & 0 & a_{15} & a_{16}
\end{bmatrix} \quad \text{or}\quad \begin{bmatrix}
0 & 0 & 1 & 0\\
0 & 0 & a_{7} & a_{8}\\
1 & 0 & 0 & 0\\
a_{13} & a_{14} & 0 & 0
\end{bmatrix}.
$$
Let $A$ be an automorphism of $G_2$ of the above form , then we obtain that
\begin{eqnarray*}
& A \in \operatorname{Aut}\left(G_2\right)_{M_1}\hspace{13cm}\\ \Leftrightarrow & A^tM_1A = M_1\hspace{13cm}\\
\Leftrightarrow & A \in \left\lbrace \begin{bmatrix}
1 & 0 & 0 & 0\\
0 & 1 & 0 & 0\\
0 & 0 & 1 & 0\\
0 & 0 & 0 & 1
\end{bmatrix}, \begin{bmatrix}
1 & 0 & 0 & 0\\
0 & 1 & 0 & 0\\
0 & 0 & 1 & 0\\
0 & 0 & 0 & -1
\end{bmatrix}, \begin{bmatrix}
1 & 0 & 0 & 0\\
0 & -1 & 0 & 0\\
0 & 0 & 1 & 0\\
0 & 0 & 0 & 1
\end{bmatrix}, \begin{bmatrix}
1 & 0 & 0 & 0\\
0 & -1 & 0 & 0\\
0 & 0 & 1 & 0\\
0 & 0 & 0 & -1
\end{bmatrix} \right\rbrace \hspace{2.4cm}\\
& \cup \left\lbrace \begin{bmatrix}
0 & 0 & 1 & 0\\
0 & 0 & 0 & 1\\
1 & 0 & 0 & 0\\
0 & 1 & 0 & 0
\end{bmatrix}, \begin{bmatrix}
0 & 0 & 1 & 0\\
0 & 0 & 0 & -1\\
1 & 0 & 0 & 0\\
0 & 1 & 0 & 0
\end{bmatrix}, \begin{bmatrix}
0 & 0 & 1 & 0\\
0 & 0 & 0 & 1\\
1 & 0 & 0 & 0\\
0 & -1 & 0 & 0
\end{bmatrix}, \begin{bmatrix}
0 & 0 & 1 & 0\\
0 & 0 & 0 & -1\\
1 & 0 & 0 & 0\\
0 & -1 & 0 & 0
\end{bmatrix} \right\rbrace \hspace{2cm}\\
\Leftrightarrow & A \in \left\langle \begin{bmatrix}
1 & 0 & 0 & 0\\
0 & 1 & 0 & 0\\
0 & 0 & 1 & 0\\
0 & 0 & 0 & -1
\end{bmatrix}, \begin{bmatrix}
0 & 0 & 1 & 0\\
0 & 0 & 0 & -1\\
1 & 0 & 0 & 0\\
0 & 1 & 0 & 0
\end{bmatrix} \right\rangle  \cong \mathbf{D}(4). \hspace{7cm}
\end{eqnarray*}
Therefore $\operatorname{Aut}\left(G_2\right)_{M_1} \cong \mathbf{D}(4)$. This the maximal group of isometric automorphisms on the Lie group $G_2$. For the metric $M_2$, it suffices to verify which of the elements of $\operatorname{Aut}\left(G_2\right)_{M_1}$ preserve the metric $M_2$. We obtain that
\begin{eqnarray*}
& A \in \operatorname{Aut}\left(G_2\right)_{M_2}\hspace{10cm}\\ \Leftrightarrow & A^tM_2A = M_2\hspace{10cm}\\
\Leftrightarrow & A \in \left\lbrace I_4, \begin{bmatrix}
1 & 0 & 0 & 0\\
0 & 1 & 0 & 0\\
0 & 0 & 1 & 0\\
0 & 0 & 0 & -1
\end{bmatrix}, \begin{bmatrix}
1 & 0 & 0 & 0\\
0 & -1 & 0 & 0\\
0 & 0 & 1 & 0\\
0 & 0 & 0 & 1
\end{bmatrix}, \begin{bmatrix}
1 & 0 & 0 & 0\\
0 & -1 & 0 & 0\\
0 & 0 & 1 & 0\\
0 & 0 & 0 & -1
\end{bmatrix} \right\rbrace \cong \left(\mathbb{Z}_2\right)^2.
\end{eqnarray*}
Similarly we can see that
$$
A \in \operatorname{Aut}\left(G_2\right)_{M_3} \Leftrightarrow A^tM_3A = M_3 \Leftrightarrow A \in \left\lbrace I_4, \operatorname{diag}\left\lbrace 1, 1, 1, -1\right\rbrace \right\rbrace \cong \mathbb{Z}_2.
$$
Finally, one can easily verify that: $A \in \operatorname{Aut}\left(G_2\right)_{M_4} \Leftrightarrow A^tM_4A = M_4 \Leftrightarrow A = I_4$.
\end{proof}
\begin{remark}
If we choose the metric $U = \begin{bmatrix}
\alpha & 0 & \gamma & 0\\
0 & 1 & 0 & 0\\
0 & 0 & \mu & 0\\
0 & 0 & 0 & 1
\end{bmatrix} \quad \alpha, \mu > 0, \quad \gamma \neq 0.$\\
Its associated symmetric positive definite matrix is $M_{\gamma} = \begin{bmatrix}
\frac{1}{\alpha^2} & 0 & \frac{-\gamma}{\alpha^2\mu} & 0\\
0 & 1 & 0 & 0\\
\frac{-\gamma}{\alpha^2\mu} & 0 & \frac{\gamma^2}{\alpha^2\mu^2} + \frac{1}{\mu^2} & 0\\
0 & 0 & 0 & 1
\end{bmatrix}.$\\
Then the group of isometric automorphisms of $M_{\gamma}$ is exactly the one of $M_1$ or $M_2$. In fact we have
$$\operatorname{Aut}\left(G_2\right)_{M_{\gamma}} \cong \left\lbrace\begin{array}{lll}
\mathbf{D}(4) \quad \text{if}\; \frac{1}{\alpha^2} = \frac{1}{\mu^2} + \frac{\gamma^2}{\alpha^2\mu^2} \\
\left(\mathbb{Z}_2\right)^2 \quad \text{if}\; \frac{1}{\alpha^2} \neq \frac{1}{\mu^2} + \frac{\gamma^2}{\alpha^2\mu^2}
\end{array}\right.
$$
If we choose the metric $U = \begin{bmatrix}
\alpha & 0 & 0 & \nu\\
0 & 1 & 0 & 0\\
0 & 0 & \mu & 0\\
0 & 0 & 0 & 1
\end{bmatrix} \quad \alpha, \mu > 0, \quad \nu \neq 0.$\\
Its associated symmetric positive definite matrix is $M_{\nu} = \begin{bmatrix}
\frac{1}{\alpha^2} & 0 & 0 & \frac{-\nu}{\alpha^2}\\
0 & 1 & 0 & 0\\
0 & 0 & \frac{1}{\mu^2} & 0\\
\frac{-\nu}{\alpha^2} & 0 & 0 & \frac{\nu^2}{\alpha^2} + 1
\end{bmatrix}.$\\
Then we have $A \in \operatorname{Aut}\left(G_2\right)_{M_{\nu}} \Leftrightarrow A^tM_{\nu}A = M_{\nu} \Leftrightarrow A \in \left\lbrace I_4, \operatorname{diag}\left\lbrace 1, -1, 1, 1\right\rbrace \right\rbrace \cong \mathbb{Z}_2$.\\
In fact, for any other left-invariant Riemannian metric on $G_2$, the group of isometric automorphisms is isomorphic to one of the groups listed in Theorem \ref{possiblegroups}.
\end{remark}
\begin{corollary}
The isometry group of the Lie group $G_2$ is given by
\begin{eqnarray*}
\operatorname{Isom}\left(G_2, M_1\right)  &\cong& G_2 \rtimes \mathbf{D}(4)\\
\operatorname{Isom}\left(G_2, M_2\right)  &\cong& G_2 \rtimes \left(\mathbb{Z}_2\right)^2\\
\operatorname{Isom}\left(G_2, M_3\right)  &\cong& G_2 \rtimes \mathbb{Z}_2\\
\operatorname{Isom}\left(G_2, M_4\right)  &\cong& G_2.
\end{eqnarray*}
\end{corollary}
\subsection{The isometry group of the Lie group $G_{3.2} \times \mathbb{R}$}
By theorem 3.3. in \cite{aitbenhaddou2025classification}, any left-invariant Riemannian metric on $\operatorname{A}_{3, 2} \oplus \operatorname{A}_1 = \operatorname{Lie}(G_{3.2} \times \mathbb{R})$ is equivalent, up to automorphism, to the following metric
$$U = \begin{bmatrix}
\alpha & 0 & 0 & \gamma\\
0 & 1 & 0 & \lambda\\
0 & 0 & \beta & \mu\\
0 & 0 & 0 & 1
\end{bmatrix} \quad \alpha, \beta > 0, \quad \lambda, \mu \geq 0, \quad \gamma \in \mathbb{R}.$$
We distinguish three cases associated with this metric
\begin{enumerate}
\item If $\lambda = \mu = \gamma = 0$, then in accordance with the bijection $\varphi$, the symmetric positive definite matrix associated to $U$ is $M_1 = (U^{-1})^t(U^{-1}) = \operatorname{diag}\left\lbrace \frac{1}{\alpha^2}, 1, \frac{1}{\beta^2}, 1\right\rbrace$.
\item If $\gamma \neq 0$ and $\lambda = \mu = 0$ in $U$, then by the bijection $\varphi$, the symmetric positive definite matrix associated to $U$ is
$$M_2 = (U^{-1})^t(U^{-1}) = \begin{bmatrix}
\frac{1}{\alpha^2} & 0 & 0 & \frac{-\gamma}{\alpha^2}\\
0 & 1 & 0 & 0\\
0 & 0 & \frac{1}{\beta^2} & 0\\
\frac{-\gamma}{\alpha^2} & 0 & 0 & \frac{\gamma^2}{\alpha^2} + 1
\end{bmatrix}.$$
\item If $\gamma \neq 0$, $\mu > 0$ and $\lambda = 0$ in $U$, then according to the bijection $\varphi$, the symmetric positive definite matrix associated to $U$ is 
$$M_3 = (U^{-1})^t(U^{-1}) = \begin{bmatrix}
\frac{1}{\alpha^2} & 0 & 0 & \frac{-\gamma}{\alpha^2}\\
0 & 1 & 0 & 0\\
0 & 0 & \frac{1}{\beta^2} & \frac{-\mu}{\beta^2}\\
\frac{-\gamma}{\alpha^2} & 0 & \frac{-\mu}{\beta^2} & \frac{\gamma^2}{\alpha^2} + \frac{\mu^2}{\beta^2} + 1
\end{bmatrix}.$$
\end{enumerate}
\begin{theorem} 
The possible groups of isometric automorphisms of $G_{3.2} \times \mathbb{R}$ are
\begin{eqnarray*}
\operatorname{Aut}\left(G_{3.2} \times \mathbb{R}\right)_{M_1}  &\cong& \left(\mathbb{Z}_2\right)^2\\
\operatorname{Aut}\left(G_{3.2} \times \mathbb{R}\right)_{M_2}  &\cong& \mathbb{Z}_2\\
\operatorname{Aut}\left(G_{3.2} \times \mathbb{R}\right)_{M_3}  &=& \{I_4\}.
\end{eqnarray*}
\end{theorem}
\begin{proof}
The automorphism group of $\operatorname{A}_{3, 2} \oplus \operatorname{A}_1 = \operatorname{Lie}(G_{3.2} \times \mathbb{R})$ consists of elements of the form \cite{christodoulakis2003automorphisms}
$$\begin{bmatrix}
a_1 & a_2 & a_3 & 0\\
0 & a_1 & a_7 & 0\\
0 & 0 & 1 & 0\\
0 & 0 & a_{15} & a_{16}
\end{bmatrix}.$$
Let $A$ be an automorphism of $\operatorname{A}_{3, 2} \oplus \operatorname{A}_1 = \operatorname{Lie}(G_{3.2} \times \mathbb{R})$ of the above form , then
\begin{eqnarray*}
& A \in \operatorname{Aut}\left(G_{3.2} \times \mathbb{R}\right)_{M_1}\hspace{10cm}\\ \Leftrightarrow & A^tM_1A = M_1\hspace{11cm}\\
\Leftrightarrow & a_2 = a_3 = a_7 = a_{15} = 0, \quad a_1 = \pm 1, \quad a_{16} = \pm 1 \hspace{5cm}\\
\Leftrightarrow & A \in \left\lbrace I_4, \begin{bmatrix}
1 & 0 & 0 & 0\\
0 & 1 & 0 & 0\\
0 & 0 & 1 & 0\\
0 & 0 & 0 & -1
\end{bmatrix}, \begin{bmatrix}
-1 & 0 & 0 & 0\\
0 & -1 & 0 & 0\\
0 & 0 & 1 & 0\\
0 & 0 & 0 & 1
\end{bmatrix}, \begin{bmatrix}
-1 & 0 & 0 & 0\\
0 & -1 & 0 & 0\\
0 & 0 & 1 & 0\\
0 & 0 & 0 & -1
\end{bmatrix} \right\rbrace \cong \left(\mathbb{Z}_2\right)^2. \hspace{0.2cm}
\end{eqnarray*}
This is the maximal group of isometric automorphisms of $G_{3.2} \times \mathbb{R}$.\\
For the metric $M_2$, it suffices to verify which of the elements of $\operatorname{Aut}\left(G_{3.2} \times \mathbb{R}\right)_{M_1}$ preserve the metric $M_2$. We obtain that
$$
A \in \operatorname{Aut}\left(G_{3.2} \times \mathbb{R}\right)_{M_2} \Leftrightarrow A^tM_2A = M_2 \Leftrightarrow A \in \left\lbrace I_4, \operatorname{diag}\left\lbrace -1, -1, 1, -1\right\rbrace \right\rbrace \cong \mathbb{Z}_2.
$$
Finally, one can easily verify that: $A \in \operatorname{Aut}\left(G_{3.2} \times \mathbb{R}\right)_{M_3} \Leftrightarrow A^tM_3A = M_3 \Leftrightarrow A = I_4$.
\end{proof}
\begin{corollary}
The isometry group of the Lie group $G_{3.2} \times \mathbb{R}$ is given by
\begin{eqnarray*}
\operatorname{Isom}\left(G_{3.2} \times \mathbb{R}, M_1\right)  &\cong& \left(G_{3.2} \times \mathbb{R}\right) \rtimes \left(\mathbb{Z}_2\right)^2\\
\operatorname{Isom}\left(G_{3.2} \times \mathbb{R}, M_2\right)  &\cong& \left(G_{3.2} \times \mathbb{R}\right) \rtimes \mathbb{Z}_2\\
\operatorname{Isom}\left(G_{3.2} \times \mathbb{R}, M_3\right)  &\cong& G_{3.2} \times \mathbb{R}.\\
\end{eqnarray*}
\end{corollary}
\subsection{The isometry group of the Lie group $G_{3.3} \times \mathbb{R}$}
By theorem 3.4. in \cite{aitbenhaddou2025classification}, any left-invariant Riemannian metric on $\operatorname{A}_{3, 3} \oplus \operatorname{A}_1 = \operatorname{Lie}(G_{3.3} \times \mathbb{R})$ is equivalent, up to automorphism, to the following metric
$$U = \begin{bmatrix}
1 & 0 & 0 & \beta\\
0 & 1 & 0 & 0\\
0 & 0 & \alpha & \gamma\\
0 & 0 & 0 & 1
\end{bmatrix} \quad \alpha > 0, \quad \beta, \gamma \geq 0.$$
We distinguish four cases associated with this metric
\begin{enumerate}
\item If $\beta = \gamma = 0$, then by the bijection $\varphi$, the symmetric positive definite matrix associated to $U$ is $M_1 = (U^{-1})^t(U^{-1}) = \operatorname{diag}\left\lbrace 1, 1, \frac{1}{\alpha^2}, 1\right\rbrace$.
\item If $\beta = 0$ and $\gamma > 0$, then based on the bijection $\varphi$, the symmetric positive definite matrix associated to $U$ is
$$M_2 = (U^{-1})^t(U^{-1}) = \begin{bmatrix}
1 & 0 & 0 & 0\\
0 & 1 & 0 & 0\\
0 & 0 & \frac{1}{\alpha^2} & \frac{-\gamma}{\alpha^2}\\
0 & 0 & \frac{-\gamma}{\alpha^2} & 1 + \frac{\gamma^2}{\alpha^2}
\end{bmatrix}.$$
\item If $\beta > 0$ and $\gamma = 0$, then according to the map $\varphi$, the symmetric positive definite matrix associated to $U$ is
$$M_3 = (U^{-1})^t(U^{-1}) = \begin{bmatrix}
1 & 0 & 0 & -\beta\\
0 & 1 & 0 & 0\\
0 & 0 & \frac{1}{\alpha^2} & 0\\
-\beta & 0 & 0 & \beta^2 + 1
\end{bmatrix}.$$
\item If $\beta, \gamma > 0$, then according to the bijection $\varphi$, the symmetric positive definite matrix associated to $U$ is 
$$M_4 = (U^{-1})^t(U^{-1}) = \begin{bmatrix}
1 & 0 & 0 & -\beta\\
0 & 1 & 0 & 0\\
0 & 0 & \frac{1}{\alpha^2} & \frac{-\gamma}{\alpha^2}\\
-\beta & 0 & \frac{-\gamma}{\alpha^2} & \beta^2 + 1 + \frac{\gamma^2}{\alpha^2}
\end{bmatrix}.$$
\end{enumerate}
\begin{theorem} 
The possible groups of isometric automorphisms of $G_{3.3} \times \mathbb{R}$ are
\begin{eqnarray*}
\operatorname{Aut}\left(G_{3.3} \times \mathbb{R}\right)_{M_1}  &=& \left(\operatorname{diag}\left\lbrace \operatorname{O}(2), 1, \pm1\right\rbrace \right)\\
\operatorname{Aut}\left(G_{3.3} \times \mathbb{R}\right)_{M_2}  &\cong& \operatorname{O}(2)\\
\operatorname{Aut}\left(G_{3.3} \times \mathbb{R}\right)_{M_3}  &\cong& \left(\mathbb{Z}_2\right)^2\\
\operatorname{Aut}\left(G_{3.3} \times \mathbb{R}\right)_{M_4}  &\cong& \mathbb{Z}_2.
\end{eqnarray*}
\end{theorem}
\begin{proof}
The automorphism group of $\operatorname{A}_{3, 3} \oplus \operatorname{A}_1 = \operatorname{Lie}(G_{3.3} \times \mathbb{R})$ consists of elements of the form \cite{christodoulakis2003automorphisms}
$$\begin{bmatrix}
a_1 & a_2 & a_3 & 0\\
a_5 & a_6 & a_7 & 0\\
0 & 0 & 1 & 0\\
0 & 0 & a_{15} & a_{16}
\end{bmatrix}.$$
Let $A$ be an automorphism of the above form; then
\begin{eqnarray*}
A \in \operatorname{O}(4) \Leftrightarrow A \in \operatorname{Aut}\left(G_{3.3} \times \mathbb{R}\right)_{M_1} &\Leftrightarrow& A^tM_1A = M_1\\
&\Leftrightarrow& \begin{bmatrix}
a_1 & a_2\\
a_5 & a_6
\end{bmatrix} \in \operatorname{O}(2), a_3 = a_7 = a_{15} = 0, a_{16} = \pm1.
\end{eqnarray*}
Hence $\operatorname{Aut}\left(G_{3.3} \times \mathbb{R}\right)_{M_1}  = \left(\operatorname{diag}\left\lbrace \operatorname{O}(2), 1, \pm1\right\rbrace \right).$ This is the maximal group of isometric autmorphisms of $G_{3.3} \times \mathbb{R}$. For the metric $M_2$, one can see that
\begin{eqnarray*}
A \in \operatorname{Aut}\left(G_{3.3} \times \mathbb{R}\right)_{M_2} &\Leftrightarrow& A^tM_2A = M_2\\
&\Leftrightarrow& \begin{bmatrix}
a_1 & a_2\\
a_5 & a_6
\end{bmatrix} \in \operatorname{O}(2), a_3 = a_7 = a_{15} = 0, a_{16} = 1.
\end{eqnarray*}
Thus $\operatorname{Aut}\left(G_{3.3} \times \mathbb{R}\right)_{M_2}  = \left(\operatorname{diag}\left\lbrace \operatorname{O}(2), 1, 1\right\rbrace \right) \cong \operatorname{O}(2).$\\
For the metric $M_3$, it suffices to verify which of the elements of $\operatorname{Aut}\left(G_{3.3} \times \mathbb{R}\right)_{M_1}$ preserve $M_3$. We consider an automorphism of the form $A = \begin{bmatrix}
a_1 & a_2 & 0 & 0\\
a_5 & a_6 & 0 & 0\\
0 & 0 & 1 & 0\\
0 & 0 & 0 & a_{16}
\end{bmatrix}$. Then
\begin{eqnarray*}
A \in \operatorname{Aut}\left(G_{3.3} \times \mathbb{R}\right)_{M_3} &\Leftrightarrow& A^tM_3A = M_3\\
&\Leftrightarrow& a_1 = a_{16} = \pm1, a_2 = 0, a_6 = \pm1, a_5 = 0
\end{eqnarray*}
Hence
\begin{eqnarray*}
\operatorname{Aut}\left(G_{3.3} \times \mathbb{R}\right)_{M_3} &=& \left\lbrace I_4, \begin{bmatrix}
1 & 0 & 0 & 0\\
0 & 1 & 0 & 0\\
0 & 0 & 1 & 0\\
0 & 0 & 0 & -1
\end{bmatrix}, \begin{bmatrix}
-1 & 0 & 0 & 0\\
0 & -1 & 0 & 0\\
0 & 0 & 1 & 0\\
0 & 0 & 0 & 1
\end{bmatrix}, \begin{bmatrix}
-1 & 0 & 0 & 0\\
0 & -1 & 0 & 0\\
0 & 0 & 1 & 0\\
0 & 0 & 0 & -1
\end{bmatrix} \right\rbrace\\
&\cong& (\mathbb{Z}_2)^2.
\end{eqnarray*}
For the metric $M_4$, we obtain that
\begin{eqnarray*}
A \in \operatorname{Aut}\left(G_{3.3} \times \mathbb{R}\right)_{M_4} &\Leftrightarrow& A^tM_4A = M_4\\
&\Leftrightarrow& a_{16} = a_1 = 1, a_2 = 0, a_6 = \pm1, a_5 = 0
\end{eqnarray*}
Thus, $\operatorname{Aut}\left(G_{3.3} \times \mathbb{R}\right)_{M_4} = \left\lbrace I_4, \operatorname{diag}\left\lbrace 1, -1, 1, 1\right\rbrace \right\rbrace \cong \mathbb{Z}_2$.
\end{proof}
\begin{corollary}
The isometry group of $G_{3.3} \times \mathbb{R}$ is given by
\begin{eqnarray*}
\operatorname{Isom}\left(G_{3.3} \times \mathbb{R}, M_{1}\right)  &\cong& \left(G_{3.3} \times \mathbb{R}\right) \rtimes \left(\operatorname{diag}\left\lbrace \operatorname{O}(2), 1, \pm1\right\rbrace \right)\\
\operatorname{Isom}\left(G_{3.3} \times \mathbb{R}, M_2\right)  &\cong& \left(G_{3.3} \times \mathbb{R}\right) \rtimes \operatorname{O}(2)\\
\operatorname{Isom}\left(G_{3.3} \times \mathbb{R}, M_3\right)  &\cong& \left(G_{3.3} \times \mathbb{R}\right) \rtimes \left(\mathbb{Z}_2\right)^2\\
\operatorname{Isom}\left(G_{3.3} \times \mathbb{R}, M_4\right)  &\cong& \left(G_{3.3} \times \mathbb{R}\right) \rtimes \mathbb{Z}_2.
\end{eqnarray*}
\end{corollary}
\subsection{The isometry group of the Lie group $G_{c} \times \mathbb{R}$}
By theorem 3.5. in \cite{aitbenhaddou2025classification}, any left-invariant Riemannian metric on $\operatorname{A}_{3, 5}^{\alpha} \oplus \operatorname{A}_1 = \operatorname{Lie}(G_{c} \times \mathbb{R})$ is equivalent, up to automorphism, to the following metric
$$U = \begin{bmatrix}
1 & \alpha & 0 & \gamma\\
0 & 1 & 0 & \lambda\\
0 & 0 & \beta & \mu\\
0 & 0 & 0 & 1
\end{bmatrix} \quad \beta > 0, \quad \alpha \in \mathbb{R}, \quad \gamma, \lambda, \mu \geq 0.$$
We distinguish four cases associated with this metric
\begin{enumerate}
\item If $\alpha = \gamma = \lambda = \mu = 0$, then according to the bijection $\varphi$, the symmetric positive definite matrix associated to $U$ is $M_1 = (U^{-1})^t(U^{-1}) = \operatorname{diag}\left\lbrace 1, 1, \frac{1}{\beta^2}, 1\right\rbrace$.
\item If $\lambda > 0$ and $\alpha = \gamma = \mu = 0$, then by the bijection $\varphi$, the symmetric positive definite matrix associated to $U$ is
$$M_2 = (U^{-1})^t(U^{-1}) = \begin{bmatrix}
1 & 0 & 0 & 0\\
0 & 1 & 0 & -\lambda\\
0 & 0 & \frac{1}{\beta^2} & 0\\
0 & -\lambda & 0 & \lambda^2 + 1
\end{bmatrix}.$$
\item If $\alpha \neq 0$, $\lambda > 0$ and $\gamma = \mu = 0$, then according to the bijection $\varphi$, the symmetric positive definite matrix associated to $U$ is 
$$M_3 = (U^{-1})^t(U^{-1}) = \begin{bmatrix}
1 & -\alpha & 0 & \alpha\lambda\\
\\
-\alpha & 1 + \alpha^2 & 0 & -\alpha^2\lambda - \lambda\\
\\
0 & 0 & \frac{1}{\beta^2} & 0\\
\\
\alpha\lambda & -\alpha^2\lambda - \lambda & 0 & \alpha^2\lambda^2+ \lambda^2 + 1
\end{bmatrix}.$$
\item If $\alpha \neq 0$, $\lambda, \mu > 0$ and $\gamma = 0$, then based on the bijection $\varphi$, the symmetric positive definite matrix associated to $U$ is 
$$M_4 = (U^{-1})^t(U^{-1}) = \begin{bmatrix}
1 & -\alpha & 0 & \alpha\lambda\\
\\
-\alpha & 1 + \alpha^2 & 0 & -\alpha^2\lambda - \lambda\\
\\
0 & 0 & \frac{1}{\beta^2} & \frac{-\mu}{\beta^2}\\
\\
\alpha\lambda & -\alpha^2\lambda - \lambda & \frac{-\mu}{\beta^2} & \alpha^2\lambda^2+ \lambda^2 + \frac{\mu^2}{\beta^2} + 1
\end{bmatrix}.$$
\end{enumerate}
\begin{theorem} 
The possible groups of isometric automorphisms of $G_{c} \times \mathbb{R}$ are
\begin{eqnarray*}
\operatorname{Aut}\left(G_{c} \times \mathbb{R}\right)_{M_1}  &\cong& \left(\mathbb{Z}_2\right)^3\\
\operatorname{Aut}\left(G_{c} \times \mathbb{R}\right)_{M_2}  &\cong& \left(\mathbb{Z}_2\right)^2\\
\operatorname{Aut}\left(G_{c} \times \mathbb{R}\right)_{M_3}  &\cong& \mathbb{Z}_2\\
\operatorname{Aut}\left(G_{c} \times \mathbb{R}\right)_{M_4}  &=& \left\lbrace I_4\right\rbrace.
\end{eqnarray*}
\end{theorem}
\begin{proof}
The automorphism group of $\operatorname{A}_{3, 5}^{\alpha} \oplus \operatorname{A}_1 = \operatorname{Lie}(G_{c} \times \mathbb{R})$ consists of elements of the form \cite{christodoulakis2003automorphisms}
$$\begin{bmatrix}
a_1 & 0 & a_3 & 0\\
0 & a_6 & a_7 & 0\\
0 & 0 & 1 & 0\\
0 & 0 & a_{15} & a_{16}
\end{bmatrix}.$$
Let $A$ be an automorphism of the above form, then
\begin{eqnarray*}
A \in \operatorname{Aut}\left(G_{c} \times \mathbb{R}\right)_{M_1} &\Leftrightarrow& A^tM_1A = M_1\\
&\Leftrightarrow& a_3 = a_7 = a_{15} = 0, a_1 = \pm1, a_6 = \pm1, a_{16} = \pm1.
\end{eqnarray*}
Hence $\operatorname{Aut}\left(G_{c} \times \mathbb{R}\right)_{M_1} = \operatorname{diag}\left\lbrace \pm1, \pm1, 1, \pm1\right\rbrace \cong \left(\mathbb{Z}_2\right)^3$. This is the maximal group of isometric automorphisms on $G_{c} \times \mathbb{R}$. For the metric $M_2$, we have
\begin{eqnarray*}
A \in \operatorname{Aut}\left(G_{c} \times \mathbb{R}\right)_{M_2} &\Leftrightarrow& A^tM_2A = M_2\\
&\Leftrightarrow& a_6 = a_{16} = \pm1, a_3 = 0, a_1 = \pm1
\end{eqnarray*}
Additionally, comparing the $(4, 3)$-th and $(3, 2)$-th components of the matrices $A^tM_2A$ and $M_2$ shows that $a_7$ and $a_{15}$ must satisfy the following system
$$\left\lbrace\begin{array}{lll}
-\lambda a_7 + (1 + \lambda^2)a_{15} = 0\\
a_7 - \lambda a_{15}  = 0
\end{array}\right.$$
This implies that $a_7 = a_{15} = 0$. Thus
\begin{eqnarray*}
\operatorname{Aut}\left(G_{c} \times \mathbb{R}\right)_{M_2} &=& \left\lbrace I_4, \begin{bmatrix}
1 & 0 & 0 & 0\\
0 & -1 & 0 & 0\\
0 & 0 & 1 & 0\\
0 & 0 & 0 & -1
\end{bmatrix}, \begin{bmatrix}
-1 & 0 & 0 & 0\\
0 & 1 & 0 & 0\\
0 & 0 & 1 & 0\\
0 & 0 & 0 & 1
\end{bmatrix}, \begin{bmatrix}
-1 & 0 & 0 & 0\\
0 & -1 & 0 & 0\\
0 & 0 & 1 & 0\\
0 & 0 & 0 & -1
\end{bmatrix} \right\rbrace\\
&\cong& (\mathbb{Z}_2)^2.
\end{eqnarray*}
There is a simple way to find this result. Since $\operatorname{Aut}\left(G_{c} \times \mathbb{R}\right)_{M_1}  \cong \left(\mathbb{Z}_2\right)^3$ is the maximal group of isometric automorphisms, it suffices to verify which of its elements preserve the metric $M_2$. For the other two metrics, we use this simple way and we obtain that
$$A \in \operatorname{Aut}\left(G_{c} \times \mathbb{R}\right)_{M_3} \Leftrightarrow A^tM_3A = M_3 \Leftrightarrow A \in \left\lbrace I_4, \operatorname{diag}\left\lbrace -1, -1, 1, -1\right\rbrace\right\rbrace \cong \mathbb{Z}_2.$$
$$A \in \operatorname{Aut}\left(G_{c} \times \mathbb{R}\right)_{M_4} \Leftrightarrow A^tM_4A = M_4 \Leftrightarrow A = I_4.$$
\end{proof}
\begin{corollary}
The isometry group of $G_{c} \times \mathbb{R}$ is given by
\begin{eqnarray*}
\operatorname{Isom}\left(G_{c} \times \mathbb{R}, M_{1}\right)  &\cong& \left(G_{c} \times \mathbb{R}\right) \rtimes \left(\mathbb{Z}_2\right)^3\\
\operatorname{Isom}\left(G_{c} \times \mathbb{R}, M_2\right)  &\cong& \left(G_{c} \times \mathbb{R}\right) \rtimes \left(\mathbb{Z}_2\right)^2\\
\operatorname{Isom}\left(G_{c} \times \mathbb{R}, M_3\right)  &\cong& \left(G_{c} \times \mathbb{R}\right) \rtimes \mathbb{Z}_2\\
\operatorname{Isom}\left(G_{c} \times \mathbb{R}, M_4\right)  &\cong& G_{c} \times \mathbb{R}.
\end{eqnarray*}
\end{corollary}
\subsection{The isometry group of the Lie group $G_{4.2}^{\alpha}$}
By theorem 4.1. in \cite{aitbenhaddou2025classification}, any left-invariant Riemannian metric on $\operatorname{A}_{4, 2}^{\alpha} = \operatorname{Lie}(G_{4.2}^{\alpha})$ is equivalent, up to automorphism, to the following metric
$$U = \begin{bmatrix}
1 & \alpha & \gamma & 0\\
0 & \beta & 0 & 0\\
0 & 0 & 1 & 0\\
0 & 0 & 0 & \lambda
\end{bmatrix} \quad \beta, \lambda > 0, \quad \alpha \geq 0, \quad \gamma \in \mathbb{R}.$$
We distinguish two cases associated with this metric
\begin{enumerate}
\item If $\alpha = \gamma = 0$, then according to the bijection $\varphi$, the symmetric positive definite matrix associated to $U$ is $M_1 = (U^{-1})^t(U^{-1}) = \operatorname{diag}\left\lbrace 1, \frac{1}{\beta^2}, 1, \frac{1}{\lambda^2}\right\rbrace$.
\item If $\alpha > 0$ and $\gamma = 0$, then according to the bijection $\varphi$, the symmetric positive definite matrix associated to $U$ is
$$M_2 = (U^{-1})^t(U^{-1}) = \begin{bmatrix}
1 & \frac{-\alpha}{\beta} & 0 & 0\\
\frac{-\alpha}{\beta} & \frac{\alpha^2 + 1}{\beta^2} & 0 & 0\\
0 & 0 & 1 & 0\\
0 & 0 & 0 & \frac{1}{\lambda^2}
\end{bmatrix}.$$
\end{enumerate}
\begin{theorem} \label{similar}
The possible groups of isometric automorphisms of $G_{4.2}^{\alpha}$ are
\begin{eqnarray*}
\operatorname{Aut}\left(G_{4.2}^{\alpha}\right)_{M_1}  &\cong& \left(\mathbb{Z}_2\right)^2\\
\operatorname{Aut}\left(G_{4.2}^{\alpha}\right)_{M_2}  &\cong& \mathbb{Z}_2.
\end{eqnarray*}
\end{theorem}
\begin{proof}
The automorphism group of $\operatorname{A}_{4, 2}^{\alpha} = \operatorname{Lie}(G_{4.2}^{\alpha})$ consists of elements of the form \cite{christodoulakis2003corrigendum}
$$\begin{bmatrix}
a_1 & 0 & 0 & a_4\\
0 & a_6 & a_7 & a_8\\
0 & 0 & a_6 & a_{12}\\
0 & 0 & 0 & 1
\end{bmatrix}.$$
Let $A$ be an automorphism of the above form, then
\begin{eqnarray*}
A \in \operatorname{Aut}\left(G_{4.2}^{\alpha}\right)_{M_1} &\Leftrightarrow& A^tM_1A = M_1\\
&\Leftrightarrow& a_4 = a_7 = a_8 = a_{12} = 0, a_1 = \pm1, a_6 = \pm1.
\end{eqnarray*}
Thus
\begin{eqnarray*}
\operatorname{Aut}\left(G_{4.2}^{\alpha}\right)_{M_1} &=& \left\lbrace I_4, \begin{bmatrix}
-1 & 0 & 0 & 0\\
0 & 1 & 0 & 0\\
0 & 0 & 1 & 0\\
0 & 0 & 0 & 1
\end{bmatrix}, \begin{bmatrix}
1 & 0 & 0 & 0\\
0 & -1 & 0 & 0\\
0 & 0 & -1 & 0\\
0 & 0 & 0 & 1
\end{bmatrix}, \begin{bmatrix}
-1 & 0 & 0 & 0\\
0 & -1 & 0 & 0\\
0 & 0 & -1 & 0\\
0 & 0 & 0 & 1
\end{bmatrix} \right\rbrace\\
&\cong& (\mathbb{Z}_2)^2.
\end{eqnarray*}
For the metric $M_2$, we obtain that 
$$A \in \operatorname{Aut}\left(G_{4.2}^{\alpha}\right)_{M_2} \Leftrightarrow A \in \left\lbrace I_4, \operatorname{diag}\left\lbrace -1, -1, -1, 1\right\rbrace \right\rbrace \cong \mathbb{Z}_2.$$
If we choose the metric $U$ where $\alpha > 0$ and $\gamma \neq 0$, its group of isometric automorphisms is still $\mathbb{Z}_2$ as above.
\end{proof}
\begin{corollary}
The isometry group of $G_{4.2}^{\alpha}$ is given by
\begin{eqnarray*}
\operatorname{Isom}\left(G_{4.2}^{\alpha}, M_{1}\right)  &\cong& G_{4.2}^{\alpha} \rtimes \left(\mathbb{Z}_2\right)^2\\
\operatorname{Isom}\left(G_{4.2}^{\alpha}, M_{2}\right)  &\cong& G_{4.2}^{\alpha} \rtimes \mathbb{Z}_2.
\end{eqnarray*}
\end{corollary}
\subsection{The isometry group of the Lie group $G_{4.2}^{1}$}
By theorem 4.2. in \cite{aitbenhaddou2025classification}, any left-invariant Riemannian metric on $\operatorname{A}_{4, 2}^{1} = \operatorname{Lie}(G_{4.2}^{1})$ is equivalent, up to automorphism, to the following metric
$$U = \begin{bmatrix}
1 & \alpha & 0 & 0\\
0 & \beta & 0 & 0\\
0 & 0 & 1 & 0\\
0 & 0 & 0 & \gamma
\end{bmatrix} \quad \beta, \gamma > 0, \quad \alpha \geq 0.$$
We distinguish two cases associated with this metric
\begin{enumerate}
\item If $\alpha = 0$, then according to the bijection $\varphi$, the symmetric positive definite matrix associated to $U$ is $M_1 = (U^{-1})^t(U^{-1}) = \operatorname{diag}\left\lbrace 1, \frac{1}{\beta^2}, 1, \frac{1}{\gamma^2}\right\rbrace$.
\item If $\alpha > 0$, then using the bijection $\varphi$, the symmetric positive definite matrix associated to $U$ is
$$M_2 = (U^{-1})^t(U^{-1}) = \begin{bmatrix}
1 & \frac{-\alpha}{\beta} & 0 & 0\\
\frac{-\alpha}{\beta} & \frac{\alpha^2 + 1}{\beta^2} & 0 & 0\\
0 & 0 & 1 & 0\\
0 & 0 & 0 & \frac{1}{\gamma^2}
\end{bmatrix}.$$
\end{enumerate}
\begin{theorem} 
The possible groups of isometric automorphisms of $G_{4.2}^{1}$ are
\begin{eqnarray*}
\operatorname{Aut}\left(G_{4.2}^{1}\right)_{M_1}  &\cong& \left(\mathbb{Z}_2\right)^2\\
\operatorname{Aut}\left(G_{4.2}^{1}\right)_{M_2}  &\cong& \mathbb{Z}_2.
\end{eqnarray*}
\end{theorem}
\begin{proof}
The automorphism group of $\operatorname{A}_{4, 2}^{1} = \operatorname{Lie}(G_{4.2}^{1})$ consists of elements of the form \cite{christodoulakis2003corrigendum}
$$\begin{bmatrix}
a_1 & 0 & a_3 & a_4\\
a_5 & a_6 & a_7 & a_8\\
0 & 0 & a_6 & a_{12}\\
0 & 0 & 0 & 1
\end{bmatrix}.$$
Let $A$ be an automorphism of the above form, then
\begin{eqnarray*}
A \in \operatorname{Aut}\left(G_{4.2}^{1}\right)_{M_1} &\Leftrightarrow& A^tM_1A = M_1\\
&\Leftrightarrow& a_3 = a_4 = a_5 = a_7 = a_8 = a_{12} = 0, a_1 = \pm1, a_6 = \pm1.
\end{eqnarray*}
Thus
\begin{eqnarray*}
\operatorname{Aut}\left(G_{4.2}^{1}\right)_{M_1} &=& \left\lbrace I_4, \begin{bmatrix}
-1 & 0 & 0 & 0\\
0 & 1 & 0 & 0\\
0 & 0 & 1 & 0\\
0 & 0 & 0 & 1
\end{bmatrix}, \begin{bmatrix}
1 & 0 & 0 & 0\\
0 & -1 & 0 & 0\\
0 & 0 & -1 & 0\\
0 & 0 & 0 & 1
\end{bmatrix}, \begin{bmatrix}
-1 & 0 & 0 & 0\\
0 & -1 & 0 & 0\\
0 & 0 & -1 & 0\\
0 & 0 & 0 & 1
\end{bmatrix} \right\rbrace\\
&\cong& (\mathbb{Z}_2)^2.
\end{eqnarray*}
This is the maximal group of isometric automorphisms of $G_{4.2}^{1}$. For the metric $M_2$, we obtain that 
$$A \in \operatorname{Aut}\left(G_{4.2}^{1}\right)_{M_2} \Leftrightarrow A \in \left\lbrace I_4, \operatorname{diag}\left\lbrace -1, -1, -1, 1\right\rbrace \right\rbrace \cong \mathbb{Z}_2.$$
\end{proof}
\begin{corollary}
The isometry group of $G_{4.2}^{1}$ is given by
\begin{eqnarray*}
\operatorname{Isom}\left(G_{4.2}^{1}, M_{1}\right)  &\cong& G_{4.2}^{1} \rtimes \left(\mathbb{Z}_2\right)^2\\
\operatorname{Isom}\left(G_{4.2}^{1}, M_{2}\right)  &\cong& G_{4.2}^{1} \rtimes \mathbb{Z}_2.
\end{eqnarray*}
\end{corollary}
\subsection{The isometry group of the Lie group $G_{4.3}$}
The automorphism group of $G_{4.3}$ and $G_{4.2}^{\alpha}$ are similar. See \cite{christodoulakis2003automorphisms} for the automorphism group of $G_{4.3}$ and see \cite{christodoulakis2003corrigendum} for the one of $G_{4.2}^{\alpha}$ (you can see also Table 7 in the end of the paper \cite{biggs2016classification}). Hence these Lie groups have similar metrics, see \cite{aitbenhaddou2025classification}. Therefore they have similar isometric automorphism groups, thus by theorem \ref{similar}, we have $\operatorname{Aut}\left(G_{4.3}\right)_{M_1}  \cong \left(\mathbb{Z}_2\right)^2$ and $\operatorname{Aut}\left(G_{4.3}\right)_{M_2}  \cong \mathbb{Z}_2$. Hence
$$\operatorname{Isom}\left(G_{4.3}, M_{1}\right)  \cong G_{4.3} \rtimes \left(\mathbb{Z}_2\right)^2, \qquad \operatorname{Isom}\left(G_{4.3}, M_{2}\right)  \cong G_{4.3} \rtimes \mathbb{Z}_2.$$
\subsection{The isometry group of the Lie group $G_{4.4}$}
By theorem 4.4. in \cite{aitbenhaddou2025classification}, any left-invariant Riemannian metric on $G_{4.4}$ is equivalent, up to automorphism, to the following metric
$$U = \begin{bmatrix}
1 & \alpha & 0 & 0\\
0 & \beta & 0 & 0\\
0 & 0 & \gamma & 0\\
0 & 0 & 0 & \lambda
\end{bmatrix} \quad \beta, \gamma, \lambda > 0, \quad \alpha \in \mathbb{R}.$$
Assume that $\alpha = 0$, then based on the bijection $\varphi$, the symmetric positive definite matrix associated with $U$ is $M_1 = (U^{-1})^t(U^{-1}) = \operatorname{diag}\left\lbrace 1, \frac{1}{\beta^2}, \frac{1}{\gamma^2}, \frac{1}{\lambda^2}\right\rbrace$. If we set $\alpha \neq 0$, we obtain a metric with group of isometric automorphisms similar to that of $M_1$.
\begin{theorem}
The group of isometric automorphisms of $G_{4.4}$ is
$\operatorname{Aut}\left(G_{4.4}\right)_{M_1}  \cong \mathbb{Z}_2$.
\end{theorem}
\begin{proof}
The automorphism group of $\operatorname{A}_{4, 4} = \operatorname{Lie}(G_{4.4})$ consists of elements of the form \cite{christodoulakis2003automorphisms}
$$\begin{bmatrix}
a_1 & a_2 & a_3 & a_4\\
0 & a_1 & a_2 & a_8\\
0 & 0 & a_1 & a_{12}\\
0 & 0 & 0 & 1
\end{bmatrix}.$$
Let $A$ be an automorphism of the above form, then
$$A \in \operatorname{Aut}\left(G_{4.4}\right)_{M_1} \Leftrightarrow A^tM_1A = M_1 \Leftrightarrow a_2 = a_3 = a_4 = a_8 = a_{12} = 0, a_1 = \pm1.$$
Hence $\operatorname{Aut}\left(G_{4.4}\right)_{M_1} = \left\lbrace I_4, \operatorname{diag}\left\lbrace -1, -1, -1, 1\right\rbrace\right\rbrace \cong \mathbb{Z}_2$.
\end{proof}
\begin{corollary}
The group of isometries of $G_{4.4}$ is given by: $\operatorname{Isom}\left(G_{4.4}, M_1\right) \cong G_{4.4} \rtimes \mathbb{Z}_2$.
\end{corollary}
\subsection{The isometry group of the Lie group $G_{4.5}^{\alpha, \beta}$}
By theorem 4.5. in \cite{aitbenhaddou2025classification}, any left-invariant Riemannian metric on $G_{4.5}^{\alpha, \beta}$ is equivalent, up to automorphism, to the following metric
$$U = \begin{bmatrix}
1 & \alpha & \beta & 0\\
0 & 1 & \gamma & 0\\
0 & 0 & 1 & 0\\
0 & 0 & 0 & \lambda
\end{bmatrix} \quad \lambda > 0, \quad \alpha, \beta \geq 0, \quad \gamma \in \mathbb{R}.$$
We distinguish three cases associated to this metric (the other cases are similar to these cases in the sens that the isometry group is isomorphic to the one of theses metrics).
\begin{enumerate}
\item If $\alpha = \beta = \gamma = 0$, then by the bijection $\varphi$, the symmetric positive definite matrix associated to $U$ is $M_1 = (U^{-1})^t(U^{-1}) = \operatorname{diag}\left\lbrace 1, 1, 1, \frac{1}{\lambda^2}\right\rbrace$.
\item If $\alpha > 0$ and $\beta = \gamma = 0$, then using the bijection $\varphi$, the symmetric positive definite matrix associated to $U$ is
$$M_2 = \left( U^{-1}\right)^t\left( U^{-1}\right) = \begin{bmatrix}
1 & -\alpha & 0 & 0\\
-\alpha & 1 + \alpha^2 & 0 & 0\\
0 & 0 & 1 & 0\\
0 & 0 & 0 & \frac{1}{\lambda^2}
\end{bmatrix}.$$
\item If $\alpha, \beta > 0$ and $\gamma \in \mathbb{R}$, then according to the map $\varphi$, the symmetric positive definite matrix associated to $U$ is
$$M_3 = (U^{-1})^t(U^{-1}) = \begin{bmatrix}
1 & -\alpha & \alpha\gamma - \beta & 0\\
\\
-\alpha & 1 + \alpha^2 & -\alpha(\alpha\gamma - \beta) - \gamma & 0\\
\\
\alpha\gamma - \beta & -\alpha(\alpha\gamma - \beta) - \gamma & (\alpha\gamma - \beta)^2 + \gamma^2 + 1 & 0\\
\\
0 & 0 & 0 & \frac{1}{\lambda^2}
\end{bmatrix}.$$
\end{enumerate}
\begin{theorem}
The group of isometric automorphisms of $G_{4.5}^{\alpha, \beta}$ is given by
\begin{eqnarray*}
\operatorname{Aut}\left(G_{4.5}^{\alpha, \beta}\right)_{M_1}  &\cong& (\mathbb{Z}_2)^3\\
\operatorname{Aut}\left(G_{4.5}^{\alpha, \beta}\right)_{M_2}  &\cong& (\mathbb{Z}_2)^2\\
\operatorname{Aut}\left(G_{4.5}^{\alpha, \beta}\right)_{M_3}  &\cong& \mathbb{Z}_2.
\end{eqnarray*}
\end{theorem}
\begin{proof}
The automorphism group of $\operatorname{A}_{4, 5}^{\alpha, \beta} = \operatorname{Lie}(G_{4.5}^{\alpha, \beta})$ consists of elements of the form \cite{christodoulakis2003automorphisms}
$$\begin{bmatrix}
a_1 & 0 & 0 & a_4\\
0 & a_6 & 0& a_8\\
0 & 0 & a_{11} & a_{12}\\
0 & 0 & 0 & 1
\end{bmatrix}.$$
Let $A$ be an automorphism of the above form, then
\begin{eqnarray*}
A \in \operatorname{Aut}\left(G_{4.5}^{\alpha, \beta}\right)_{M_1} &\Leftrightarrow& A^tM_1A = M_1\\
&\Leftrightarrow& a_4 = a_8 = a_{12} = 0, a_1 = \pm1, a_6 = \pm1, a_{11} = \pm1.
\end{eqnarray*}
Hence $\operatorname{Aut}\left(G_{4.5}^{\alpha, \beta}\right)_{M_1} = \operatorname{diag}\left\lbrace \pm1, \pm1, \pm1, 1\right\rbrace \cong \left(\mathbb{Z}_2\right)^3$.\\
For the metric $M_2$, it suffices to verify which of the elements of $\operatorname{Aut}\left(G_{4.5}^{\alpha, \beta}\right)_{M_1}$ preserve $M_2$. We obtain that
\begin{eqnarray*}
\operatorname{Aut}\left(G_{4.5}^{\alpha, \beta}\right)_{M_2} &=& \left\lbrace I_4, \begin{bmatrix}
1 & 0 & 0 & 0\\
0 & 1 & 0 & 0\\
0 & 0 & -1 & 0\\
0 & 0 & 0 & 1
\end{bmatrix}, \begin{bmatrix}
-1 & 0 & 0 & 0\\
0 & -1 & 0 & 0\\
0 & 0 & 1 & 0\\
0 & 0 & 0 & 1
\end{bmatrix}, \begin{bmatrix}
-1 & 0 & 0 & 0\\
0 & -1 & 0 & 0\\
0 & 0 & -1 & 0\\
0 & 0 & 0 & 1
\end{bmatrix} \right\rbrace\\
&\cong& (\mathbb{Z}_2)^2.
\end{eqnarray*}
Finally, one gets that
$$A \in \operatorname{Aut}\left(G_{4.5}^{\alpha, \beta}\right)_{M_3} \Leftrightarrow A^tM_3A = M_3 \Leftrightarrow A \in \left\lbrace I_4, \begin{bmatrix}
-1 & 0 & 0 & 0\\
0 & -1 & 0 & 0\\
0 & 0 & -1 & 0\\
0 & 0 & 0 & 1
\end{bmatrix} \right\rbrace \cong \mathbb{Z}_2.$$
\end{proof}
\begin{corollary}
The isometry group of $G_{4.5}^{\alpha, \beta}$ is given by
\begin{eqnarray*}
\operatorname{Isom}\left(G_{4.5}^{\alpha, \beta}, M_1\right)  &\cong& G_{4.5}^{\alpha, \beta} \rtimes (\mathbb{Z}_2)^3\\
\operatorname{Isom}\left(G_{4.5}^{\alpha, \beta}, M_2\right)  &\cong& G_{4.5}^{\alpha, \beta} \rtimes (\mathbb{Z}_2)^2\\
\operatorname{Isom}\left(G_{4.5}^{\alpha, \beta}, M_3\right)  &\cong& G_{4.5}^{\alpha, \beta} \rtimes \mathbb{Z}_2.
\end{eqnarray*}
\end{corollary}
\subsection{The isometry group of the Lie group $G_{4.7}$}
By theorem 4.7. in \cite{aitbenhaddou2025classification}, any left-invariant Riemannian metric on $G_{4.7}$ is equivalent, up to automorphism, to the following metric
$$U = \begin{bmatrix}
\alpha & \beta & \gamma & 0\\
0 & 1 & 0 & 0\\
0 & 0 & \lambda & 0\\
0 & 0 & 0 & \mu
\end{bmatrix} \quad \alpha, \lambda, \mu > 0, \quad \beta \geq 0, \quad \gamma \in \mathbb{R}.$$
We distinguich two cases associated with this metric
\begin{enumerate}
\item If $\beta = \gamma = 0$, then based on the bijection $\varphi$, the symmetric positive definite matrix associated with $U$ is $M_1 = (U^{-1})^t(U^{-1}) = \operatorname{diag}\left\lbrace \frac{1}{\alpha^2}, 1, \frac{1}{\lambda^2}, \frac{1}{\mu^2}\right\rbrace$.
\item If $\beta \neq 0$ and $\gamma = 0$, then using the map $\varphi$, the symmetric positive definite matrix associated with $U$ is 
$$M_2 = (U^{-1})^t(U^{-1}) = \begin{bmatrix}
\frac{1}{\alpha^2} & \frac{-\beta}{\alpha^2} & 0 & 0\\
\frac{-\beta}{\alpha^2} & 1 + \frac{\beta^2}{\alpha^2} & 0 & 0\\
0 & 0 & \frac{1}{\lambda^2} & 0\\
0 & 0 & 0 & \frac{1}{\mu^2}
\end{bmatrix}.$$
\end{enumerate}
\begin{theorem}
The group of isometric automorphisms of $G_{4.7}$ is given by:\\
$\operatorname{Aut}\left(G_{4.7}\right)_{M_1}  \cong \mathbb{Z}_2, \qquad \operatorname{Aut}\left(G_{4.7}\right)_{M_2} = \left\lbrace I_4\right\rbrace$.
\end{theorem}
\begin{proof}
The automorphism group of $\operatorname{A}_{4, 7} = \operatorname{Lie}(G_{4.7})$ consists of elements of the form \cite{christodoulakis2003automorphisms}
$$\begin{bmatrix}
a_6^2 & -a_{12}a_6 & -a_{12}(a_6 + a_7) + a_6a_8 & a_4\\
0 & a_6 & a_7 & a_8\\
0 & 0 & a_6 & a_{12}\\
0 & 0 & 0 & 1
\end{bmatrix}.$$
Let $A$ be an automorphism of the above form, then
$$A \in \operatorname{Aut}\left(G_{4.7}\right)_{M_1} \Leftrightarrow A^tM_1A = M_1 \Leftrightarrow a_4 = a_7 = a_8 = a_{12} = 0, a_6 = \pm1.$$
Hence $\operatorname{Aut}\left(G_{4.7}\right)_{M_1} = \left\lbrace I_4, \operatorname{diag}\left\lbrace 1, -1, -1, 1\right\rbrace\right\rbrace \cong \mathbb{Z}_2$.\\
For the metric $M_2$, one gets that $A \in \operatorname{Aut}\left(G_{4.7}\right)_{M_2} \Leftrightarrow A^tM_2A = M_2 \Leftrightarrow A = I_4$.
\end{proof}
\begin{corollary}
The group of isometries of $G_{4.7}$ is given by:\\ $\operatorname{Isom}\left(G_{4.7}, M_1\right) \cong G_{4.7} \rtimes \mathbb{Z}_2, \quad \operatorname{Isom}\left(G_{4.7}, M_2\right) \cong G_{4.7}$.
\end{corollary}
\subsection{The isometry group of the Lie group $G_{4.8}^{\alpha}$}
Since there is a parameter $\beta$ on the automorphisms of $\operatorname{A}_{4, 9}^{\beta} = \operatorname{Lie}(G_{4.8}^{\alpha})$; We will not use $\beta$ in the parameters of the metrics on $G_{4.8}^{\alpha}$. Note that there is no relationship between the parameters used in the names of our Lie groups and the parameters on the metrics. By theorem 4.8. in \cite{aitbenhaddou2025classification}, any left-invariant Riemannian metric on $G_{4.8}^{\alpha}$ is equivalent, up to automorphism, to the following metric
$$U = \begin{bmatrix}
\alpha & \eta & \gamma  & 0\\
0 & 1 & \lambda & 0\\
0 & 0 & 1 & 0\\
0 & 0 & 0 & \mu
\end{bmatrix} \quad \alpha, \mu > 0, \quad \eta, \gamma \geq 0, \quad \lambda \in \mathbb{R}.$$
We distinguish three cases associated to this metric
\begin{enumerate}
\item If $\eta = \gamma = \lambda = 0$, by the map $\varphi$, the symmetric positive definite matrix associated to $U$ is $M_1 = (U^{-1})^t(U^{-1}) = \operatorname{diag}\left\lbrace \frac{1}{\alpha^2}, 1, 1, \frac{1}{\mu^2}\right\rbrace$.
\item If $\eta > 0$ and $\gamma = \lambda = 0$, then using the bijection $\varphi$, the symmetric positive definite matrix associated to $U$ is
$$M_2 = \left( U^{-1}\right)^t\left( U^{-1}\right) = \begin{bmatrix}
\frac{1}{\alpha^2} & \frac{-\eta}{\alpha^2} & 0 & 0\\
\frac{-\eta}{\alpha^2} & 1 + \frac{\eta^2}{\alpha^2} & 0 & 0\\
0 & 0 & 1 & 0\\
0 & 0 & 0 & \frac{1}{\mu^2}
\end{bmatrix}.$$
\item If $\eta > 0$, $\lambda \neq 0$ and $\gamma = 0$, by $\varphi$, the symmetric positive definite matrix associated to $U$ is
$$M_3 = \left( U^{-1}\right)^t\left( U^{-1}\right) = \begin{bmatrix}
\frac{1}{\alpha^2} & \frac{-\eta}{\alpha^2} & \frac{\eta\lambda}{\alpha^2} & 0\\
\\
\frac{-\eta}{\alpha^2} & 1 + \frac{\eta^2}{\alpha^2} & \frac{-\eta^2\lambda}{\alpha^2} - \lambda & 0\\
\\
\frac{\eta\lambda}{\alpha^2} & \frac{-\eta^2\lambda}{\alpha^2} - \lambda & 1 & 0\\
\\
0 & 0 & 0 & \frac{1}{\mu^2}
\end{bmatrix}.$$
\end{enumerate}
\begin{theorem}
The group of isometric automorphisms of $G_{4.8}^{\alpha}$ is given by
\begin{eqnarray*}
\operatorname{Aut}\left(G_{4.8}^{\alpha}\right)_{M_1}  &\cong& (\mathbb{Z}_2)^2\\
\operatorname{Aut}\left(G_{4.8}^{\alpha}\right)_{M_2}  &\cong& \mathbb{Z}_2\\
\operatorname{Aut}\left(G_{4.8}^{\alpha}\right)_{M_3}  &=& \left\lbrace I_4\right\rbrace.
\end{eqnarray*}
\end{theorem}
\begin{proof}
The automorphism group of $\operatorname{A}_{4, 9}^{\beta} = \operatorname{Lie}(G_{4.8}^{\alpha})$ consists of elements of the form \cite{christodoulakis2003automorphisms}
$$\begin{bmatrix}
a_{11}a_6 & -a_{12}a_6/\beta & a_8a_{11}& a_4\\
0 & a_6 & 0 & a_8\\
0 & 0 & a_{11} & a_{12}\\
0 & 0 & 0 & 1
\end{bmatrix}.$$
Let $A$ be an automorphism of the above form, then
\begin{eqnarray*}
A \in \operatorname{Aut}\left(G_{4.8}^{\alpha}\right)_{M_1} &\Leftrightarrow& A^tM_1A = M_1\\
&\Leftrightarrow& a_4 = a_8 = a_{12} = 0, a_6 = \pm1, a_{11} = \pm1.
\end{eqnarray*}
Hence 
\begin{eqnarray*}
\operatorname{Aut}\left(G_{4.8}^{\alpha}\right)_{M_1} &=& \left\lbrace I_4, \begin{bmatrix}
-1 & 0 & 0 & 0\\
0 & 1 & 0 & 0\\
0 & 0 & -1 & 0\\
0 & 0 & 0 & 1
\end{bmatrix}, \begin{bmatrix}
-1 & 0 & 0 & 0\\
0 & -1 & 0 & 0\\
0 & 0 & 1 & 0\\
0 & 0 & 0 & 1
\end{bmatrix}, \begin{bmatrix}
1 & 0 & 0 & 0\\
0 & -1 & 0 & 0\\
0 & 0 & -1 & 0\\
0 & 0 & 0 & 1
\end{bmatrix} \right\rbrace\\
&\cong& (\mathbb{Z}_2)^2.
\end{eqnarray*}
For the metric $M_2$, we obtain that
$$A \in \operatorname{Aut}\left(G_{4.8}^{\alpha}\right)_{M_2} \Leftrightarrow A^tM_2A = M_2 \Leftrightarrow A \in \left\lbrace I_4, \begin{bmatrix}
-1 & 0 & 0 & 0\\
0 & -1 & 0 & 0\\
0 & 0 & 1 & 0\\
0 & 0 & 0 & 1
\end{bmatrix} \right\rbrace \cong \mathbb{Z}_2.$$
Finally, one can see that $A \in \operatorname{Aut}\left(G_{4.8}^{\alpha}\right)_{M_3} \Leftrightarrow A^tM_3A = M_3 \Leftrightarrow A \in \left\lbrace I_4 \right\rbrace.$
\end{proof}
\begin{corollary}
The isometry group of $G_{4.8}^{\alpha}$ is given by
\begin{eqnarray*}
\operatorname{Isom}\left(G_{4.8}^{\alpha}, M_1\right)  &\cong& G_{4.8}^{\alpha} \rtimes (\mathbb{Z}_2)^2\\
\operatorname{Isom}\left(G_{4.8}^{\alpha}, M_2\right)  &\cong& G_{4.8}^{\alpha} \rtimes \mathbb{Z}_2\\
\operatorname{Isom}\left(G_{4.8}^{\alpha}, M_3\right)  &\cong& G_{4.8}^{\alpha}.
\end{eqnarray*}
\end{corollary}
\textbf{Open question:}\\
The study of the isometry group of left-invariant Riemannian metrics on the following Lie groups is still an open question:
\begin{enumerate}
\item $\operatorname{Sol}_{\mu}^4$, $\widetilde{\operatorname{Isom}_0(\mathbb{R}^2)} \times \mathbb{R}$ (unimodular solvable not of type $(R)$, see \cite{van2017metrics}).
\item $\widetilde{\operatorname{SL}(2,\mathbb{R})} \times \mathbb{R}$, $\operatorname{SU}(2, \mathbb{R}) \times \mathbb{R}$ (unimodular non-solvable, see \cite{van2017metrics}).
\item The simply connected Lie groups associated with the following Lie algebras:\\ $\operatorname{A}_{3, 7}^{\alpha} \oplus \operatorname{A}_1, \operatorname{A}_{4, 6}^{\alpha, \beta}, \operatorname{A}_{4, 11}^{\alpha}$ and $\operatorname{A}_{4, 12}$. They are nonunimodular and not of type $(R)$, as shown in this paper.
\end{enumerate}
The classification of left-invariant Riemannian metrics on these Lie groups is given by \cite{van2017metrics,aitbenhaddou2025classification}. The question is a bit complicated, but one can use the fact that the isotropy subgroup $\operatorname{Isom}(G, g)_e$ can be injected into the intersection $\operatorname{O}(\mathfrak{g}, g) \cap \mathcal{C}(\hat{r})$, where $\operatorname{O}(\mathfrak{g}, g)$ is the orthogonal group of $\mathfrak{g}$ with respect to $g$, and $\mathcal{C}(\hat{r})$ is the subspace of matrices that commute with the Ricci transformation $\hat{r}$ of $(G, g)$. For more details, see the strategy used in the following paper \cite{aitbenhaddou2024isometry}. The signature of the Ricci operator of left-invariant Riemannian metrics on four-dimensional Lie groups is studied in \cite{kremlev2009signature,kremlev2010signature}.

\end{document}